\documentclass{amsart}
\usepackage{amsmath,amsthm,amssymb,isomath,mathrsfs,calligra}
\usepackage{derivative}
\usepackage[new]{old-arrows}
\usepackage{graphicx}
\usepackage{float}
\usepackage[english]{babel}
\usepackage{latexsym}
\usepackage[utf8]{inputenc}
\usepackage{babel}
\usepackage{enumerate}
\usepackage{epsfig}
\usepackage{relsize}
\usepackage{layout}
\usepackage{mathtools}
\usepackage[bottom]{footmisc}
\usepackage[
margin=2cm,
includefoot,
footskip=30pt,
]{geometry}
\usepackage{layout}
\usepackage{tikz}
\usepackage{tikz-cd}
\usetikzlibrary{cd}
\usepackage{derivative}
\usepackage[all]{xy}
\usepackage{comment}
\usepackage{fancyhdr}
\usepackage{xcolor,hyperref}
\hypersetup{colorlinks=true, breaklinks,
	urlcolor=magenta, linkcolor = purple,
	anchorcolor=purple, citecolor=blue}
\usepackage{csquotes}
\usepackage{xcolor,hyperref}
\hypersetup{colorlinks=true, breaklinks,
	urlcolor=magenta, linkcolor = red,
	anchorcolor=purple, citecolor=blue}
\usepackage{comment}
\usepackage{bm}

\usepackage{lmodern}   
\usepackage[T1]{fontenc}

\newcommand{\K}{\mathcal{K}}
\newcommand{\C}{\mathbf{C}}
\renewcommand{\P}{\mathbf{P}}

\renewcommand{\phi}{\varphi}
\renewcommand{\O}{\mathcal{O}}

\newcommand{\I}{\mathscr{I}}
\newcommand{\F}{\mathcal{F}}
\newcommand{\g}{\mathcal{G}}
\newcommand{\E}{\mathcal{E}}
\newcommand{\Z}{\mathbf{Z}}

\newcommand{\Sym}{\mathrm{Sym}}
\renewcommand{\S}{\mathrm{S}}

\newcommand{\Pic}{\mathrm{Pic}}
\newcommand{\G}{\Gamma}
\renewcommand{\l}{\ell}

\newcommand{\w}{\widetilde}
\newcommand{\OPn}{\mathcal{O}_{\mathbf{P}^{n}}}
\newcommand{\OPN}{\mathcal{O}_{\mathbf{P}^{N}}}
\renewcommand{\o}{\otimes}

\newcommand{\eps}{\varepsilon}

\newcommand{\m}{\mathfrak{m}}

\newcommand{\s}{\sigma}
\renewcommand{\to}{\rightarrow}

\newcommand{\reg}{\mathrm{reg}}
\renewcommand{\H}{\mathcal{H}}
\newcommand{\OPE}{\O_{\P(\E)}}
\newcommand{\mr}{\mathrm}
\newcommand{\mc}{\mathcal}
\newcommand{\mb}{\mathbf}

\renewcommand{\L}{\Lambda}
\newcommand{\EExt}{\mathscr{E}{\kern -2pt{xt}}}
\newcommand{\Tor}{\mathscr{T}\text{\kern -3pt {\calligra\large or}}\,}
\newcommand{\RHom}{\mathbf{R}{\kern -0.3pt {\mathscr{H}}}\text{\kern -4.5pt {\calligra\large om}}\,}
\newcommand{\Hom}{\mathscr{H}\text{\kern -4.5pt {\calligra\large om}}\,}

\DeclareMathOperator{\rk}{rk}

\DeclareMathOperator{\Spec}{Spec}

\DeclareMathOperator{\depth}{depth}

\theoremstyle{plain}
\newtheorem{thm}{Theorem}[section]
\theoremstyle{plain}
\newtheorem{bthm}{Theorem}
\theoremstyle{plain}

\theoremstyle{definition}
\newtheorem{defi}[thm]{Definition}
\theoremstyle{plain}
\newtheorem{prop}[thm]{Proposition}
\theoremstyle{plain}
\newtheorem{cor}[thm]{Corollary}
\theoremstyle{plain}

\theoremstyle{plain}
\newtheorem{lemma}[thm]{Lemma}
\theoremstyle{definition}
\newtheorem{rmk}[thm]{Remark}
\theoremstyle{plain}

\theoremstyle{plain}

\theoremstyle{definition}
\newtheorem{ex}[thm]{Example}
\theoremstyle{definition}
\newtheorem{notat}[thm]{Notation}

\numberwithin{equation}{section}

\title[On the projective normality of Ulrich bundles on some low-dimensional varieties]{On the projective normality of Ulrich bundles\\ on some low-dimensional varieties}
\author[Valerio Buttinelli]{Valerio Buttinelli*}
\address{Dipartimento di Matematica "Guido Castelnuovo"\\
	Sapienza Universit\`a di Roma\\
	Piazzale Aldo Moro 5, 00185 Roma, Italy}
\thanks{*Work produced as a part of the author's PhD thesis under the supervision of Angelo Felice Lopez}
\email[Valerio~Buttinelli]{valerio.buttinelli@uniroma1.it}

\begin{document}
	
	\begin{abstract}
	 We study the projective normality of the projective bundle of an Ulrich vector bundle embedded through the complete linear system of its tautological line bundle. The focus will be on Ulrich bundles defined over curves, surfaces with $q=p_g=0$ and hypersurfaces of dimension $2$ and $3.$
	\end{abstract}
	
	\maketitle
	
	\section{Introduction} 
	
	A class of vector bundles which gained more and more attention in recent years is the one of Ulrich bundles: given a polarized smooth projective variety $(X,B)$ with $B$ globally generated, a vector bundle $\E$ on $X$ is $B$-Ulrich if $H^i(X,\E(-pB))=0$ for all $i\ge0$ and $1\le p\le \dim X.$ The importance of Ulrich bundles in algebraic geometry is well-known and mainly comes from the several implications determined by their existence, which is still widely conjectural, in Boij-Söderberg
	theory and determinantal representations of Chow forms (see \cite{eisenbud2003resultants,coskun2017survey,beauville2018introduction,costa2021ulrich} for a detailed overview). 
	
	Among the other properties, Ulrich bundles are globally generated which means that they always posses a certain positivity. In fact, when $B$ induces an étale morphism onto its schematic image (e.g. if $B$ is very ample), Lopez-Sierra theorem \cite[Theorem 1]{lopez2023geometrical} and \cite[Corollary 3 \& Theorem 4]{buttinelli2024positivity} tell that an Ulrich bundle $\E$ is ample if and only if it is very ample if and only if either $X$ does not contain lines or $\E_{|L}$ is ample on every line $L\subset X.$ It is then natural trying to understand the embedding of the corresponding projective bundle through the (complete) linear system of the tautological line bundle. 
	
	A very classical and relevant question is determining when a (linear system of a) very ample line bundle embeds the underlying variety as a projectively normal scheme (in some projective space). In this regard, this paper is devoted to the study of the projective normality of an Ulrich bundle, that is by definition the normal generation of its tautological line bundle. As it will be clear in the following, an unified result as Lopez-Sierra theorem for very ampleness appears out of reach. Indeed, even if an Ulrich bundle is very ample, it is not always projectively normal. On curves this happens if the degree of the polarization is big with respect to the genus, with an optimal bound in some cases. However this is no longer true in higher dimension, for instance on hypersurfaces where the behaviour of projective normality of Ulrich bundles suggests that it is unlikely to get a general criterion. The main results, which will be on low-dimensional varieties where at least Calstelnuovo-Mumford regularity is well-behaved with respect to tensor operations, are the following.
	
	\begin{bthm}\label{thm:proj-norm-curve}
		Let $C$ be a smooth projective curve of genus $g$ and let $B$ be a globally generated ample line bundle of degree $d$ on $C.$ Let $\E$ be a $B$-Ulrich bundle on $C.$ Then:
		\begin{itemize}
			\item[(a)] $\E$ is projectively normal if $d>g+1.$ 
			\item[(b)] $\E$ satisfies $(N_1)$ and $\OPE(1)$ is Koszul if $d>g+2.$
			\item[(c)] $\E$ satisfies $(N_p)$ for $p\ge 2$ if $d>\frac{1}{2} \left((g+p+1)+\sqrt{g^2+2g(3p+1)+(p-1)^2}\right).$
			\item[(d)] If there exists a linear series $|V|\subseteq|B|$ which induces a morphism which is étale onto the schematic image, the general $B$-Ulrich bundle of rank $r$ on $C$ is projectively normal as soon as $C$ supports a non-special normally generated line bundle of degree $d.$ This holds in particular if $d\ge g+2-\mr{Cliff}(C).$
			\item[(e)] If $C$ is general of genus $g\ge3$ and $B$ is a general very ample line bundle of degree $$d\ge\frac{3+\sqrt{8g+1}}{2},$$ then the general $B$-Ulrich bundle of rank $r$ is projectively normal. Moreover this bound is sharp for $r=1.$
		\end{itemize} 
	\end{bthm}
	
	\begin{bthm}\label{thm:proj-norm-0-regular-q=p=0}
		Let $S\subset\P^N$ be a smooth projective surface with $q(S)=p_g(S)=0$ and let $\E$ be an ample $0$-regular vector bundle of rank $r\ge 2$ on $S$ such that $h=h^0(S,\E)\ge r+3.$ Let $E=\det(\E)$ be the determinant bundle and let $\l=\binom{h-r}{2}-1.$ The following are equivalent:
		\begin{itemize}
			\item[(1)] $\P(\E)$ is not aCM.
			\item[(2)] $\E$ is not projectively normal.
			\item[(3)] There exist a closed subscheme $Z\subset S$ and a non-zero divisor $D\subset S$ such that:
			\begin{itemize}
				\item[(a)] $Z$ is smooth of dimension $0.$
				\item[(b)] $Z$ is the degeneracy locus of $\l$ general sections $s_1,\dots,s_\l\in H^0(S,\L^2M_\E^\ast).$
				\item[(c)] $[Z]=\frac{1}{2}(h-r-2)\left((h-r+1)c_1(\E)^2-2c_2(\E)\right).$
				\item[(d)] $D\in |K_S+(h-r-1)E|.$
				\item[(e)] $Z\subset D.$
			\end{itemize}
			\item[(4)] There exist a closed subscheme $Z\subset S$ and a curve $C\subset S$ such that: 
			\begin{itemize}
				\item[(f)] $Z$ is the degeneracy locus of $\l$ general sections $\s_1,\dots,\s_\l\in  H^0(S,\L^2M_\E^\ast).$
				\item[(g)] $C$ is the degeneracy locus of the $(\l+1)$ general sections $\s_1,\dots,\s_\l,\s_{\l+1}\in  H^0(S,\L^2M_\E^\ast).$
				\item[(h)] $C\in |(h-r-1)E|$ is smooth and irreducible.
				\item[(i)] $Z\subset C$ is a special (effective) divisor.
			\end{itemize}
		\end{itemize}
	\end{bthm}
	
	\begin{bthm}\label{thm:projective-normality-hypersurface}
		Let $X\subset\P^{n+1}$ be a smooth hypersurface of degree $d\ge3$ with $2\le n\le 3$ and let $\E$ be an Ulrich bundle of rank $r$ on $X.$ Let $$\mu_\E\colon H^0(X,\E)\o H^0(X,\E)\to H^0(X,\E\o\E)$$ denote the multiplication of sections. Then:
		\begin{itemize}
			\item[(a)] If $n=2$ and $\det(\E)=\O_X(\frac{r}{2}(d-1)),$ then $\mu_\E$ cannot be surjective and $\E$ cannot be projectively normal if $d\ge 5,$ or $d=4$ and $r\le5,$ or $d=3$ and $r\le 2.$
			\item[(b)] If $n=3$ and $d\ge4,$ then $\mu_\E$ is never surjective and $\E$ cannot be projectively normal if $r>\frac{d+4}{3}.$
		\end{itemize}
	\end{bthm}
	
	As mentioned above, the behavior of the projective normality on (low-dimensional) hypersurfaces is the most unexpected: a general hypersurface contains no lines if its degree is greater than the double of its dimension, therefore Ulrich bundles are expected to be always very positive (at least in this situation). Theorem \ref{thm:projective-normality-hypersurface} and Remark \ref{rmk:hyper} seem to suggest that just low-rank Ulrich bundles are (potentially) projectively normal. However, by Buchweitz-Greuel-Schreyer conjecture, which has been proved for the general hypersurface \cite{erman2021matrix}, the rank of Ulrich bundles is expected to be grater than or equal to $2^{\lfloor\frac{n-2}{2}\rfloor},$ where $n$ is the dimension of the hypersurface (see also \cite{lopez2024ulrich} for the non-existence of low-rank Ulrich bundles on hypersurfaces). Therefore, against the expectations, it seems that Ulrich bundles on hypersurfaces are rarely projectively normal.\\
	\\
	{\bf Acknowledgments.} I would like to thank my PhD advisor, Prof. Angelo Felice Lopez, for his extremely valuable guidance throughout the preparation of this paper.
	
	\section{Notations and basic facts on Ulrich bundles}
	Throughout the paper we will adopt the following conventions unless otherwise specified:
	\begin{itemize}
		\setlength\itemsep{0em}
		\item All schemes are separated and of finite type over the field of complex numbers $\C.$ A \emph{curve} and a \emph{surface} are connected equidimensional schemes of dimension $1$ and $2$ respectively.
		\item A variety is an integral scheme and subvarieties are always closed. 
		\item For a smooth variety $X,$ we write $q(X)=h^1(X,\O_X)$ for the \emph{irregularity} and $p_g(X)=h^0(X,K_X)$ for the \emph{geometric genus.} We say $X$ is \emph{regular} if $q(X)=0.$
		\item A point $x\in X$ is always closed and its ideal sheaf is denoted by $\m_x.$ The ideal sheaf of a closed subscheme $Y\subset X$ is denoted by $\I_{Y/X}.$
		\item Given a line bundle $L,$ we write $\F(pL)=\F\o L^{\o p}$ for any sheaf $\F$ and any $p\in\Z.$
		\item For an embedded projective scheme $X\subset\P^N,$ we denote by $I_{X/\P^N}=\bigoplus_{t\in\Z}H^0(\P^N,\I_{X/\P^N}(t))$ the \emph{homogeneous saturated ideal} of $X$ in $\P^N$ and by $R_X=\C[x_0,\dots,x_N]/I_{X/\P^N}$ the \emph{homogeneous coordinate ring} of $X.$
		\item For a vector bundle $\E$ on a scheme $X$ we set $\P(\E)=\mr{Proj}(\Sym(\E))$ for the projective bundle of $\E$ and we denote by $\pi\colon\P(\E)\to X$ the natural projection.
		\item A property is \emph{general} if it holds in the complement of a proper (Zariski) closed subset. A property is \emph{very general} if it is satisfied off a countable union of proper (Zariski) closed subsets.
	\end{itemize}
	
	\begin{defi}
		Let $\E$ be a globally generated vector bundle on a scheme $X.$ The \emph{syzygy bundle of $\E$} is $$M_\E=\ker\left(H^0(X,\E)\o\O_X\to\E\right)$$ and we call the exact sequence $$\begin{tikzcd}
			0\rar&M_\E\rar&H^0(X,\E)\o\O_X\rar&\E\rar&0,
		\end{tikzcd}$$ the \emph{syzygy exact sequence} of $\E.$
	\end{defi}
	
	\begin{defi}
		A vector bundle $\E$ on a projective scheme $X$ is \emph{(very) ample} on $X$ if $\OPE(1)$ is (very) ample on $\P(\E).$ When $\E$ is very ample, we always consider the embedding $\P(\E)\hookrightarrow\P(H^0(X,\E))$ induced by the complete linear system $|\OPE(1)|.$
	\end{defi}
	
	We now come to some generalities on Ulrich bundles.
	
	\begin{defi}
		Let $X$ be a smooth projective variety and let $B$ be a globally generated ample line bundle on $X.$ A vector bundle $\E$ on $X$ is \emph{$B$-Ulrich} if $H^i(X,\E(-pB))=0$ for all $i\ge0$ and $1\le p\le \dim X.$ If $B$ is very ample defining an embedding $X\subset\P^N,$ we simply say that $\E$ is Ulrich on $X\subset\P^N.$
	\end{defi}
	
	We recall the main properties of $B$-Ulrich bundles. For proofs we refer to \cite{coskun2017survey,beauville2018introduction,costa2021ulrich} or to \cite[Lemma 3.9]{buttinelli2024positivity} for the ample and free case.
	
	\begin{rmk}\label{rmk:ulrich-facts}
		Let $X$ be a smooth projective variety of dimension $n\ge1$ and let $B$ be a globally generated ample line bundle on $X$ with $B^n=d.$ Then a $B$-Ulrich bundle $\E$ of rank $r$ is semistable, aCM and $0$-regular (in the sense of Castelnuovo-Mumford). In particular $\E$ is globally generated and $h^0(X,\E)=rd.$ Moreover $c_1(\E)\cdot B^{n-1}=\frac{r}{2}(K_X+(n+1)B)$ which means that $c_1(\E)=\frac{r}{2}(K_X+(n+1)B)$ if $\Pic(X)\cong\Z.$ Observe also that $\O_X(kB)$ is $B$-Ulrich if and only if $(X,B,k)=(\P^{n},\OPn(1),0).$
	\end{rmk}
	
	Some of these facts will be used without further reference.

	\section{Castelnuovo-Mumford regularity of tensor product}
	In this section we are going to recollect some known results on the Castelnuovo-Mumford regularity of the tensor product of coherent sheaves, focusing the attention on vector bundles. We start by recalling the definition. For a detailed account, see \cite[§1.8]{lazarsfeld2017positivity}.
	
	\begin{defi}
		Let $X$ be a projective variety and let $B$ a globally generated ample line bundle on $X.$ A coherent sheaf $\F$ on $X$ is \emph{$m$-regular} with respect to $B$ if $H^i(X,\F((m-i)B))=0$ for all $i>0.$ The \emph{regularity of $\F$ (with respect to $B$)} is 
		$$\reg_B(\F):=\min\left\{s\in\Z\ |\ \F\ \text{is}\ s\text{-regular with respect to}\ B\right\}.$$
	\end{defi}
	
	As is well-known, e.g. from \cite[Proposition 1.8.9]{lazarsfeld2017positivity}, the (Castelnuovo-Mumford) regularity of the tensor product of two vector bundles on the projective space is (at most) the sum of the regularity of each vector bundle. This no longer holds for other varieties, mainly because the polarization is not $(-1)$-regular (with respect to itself). In \cite[§3]{arapura2004frobenius}, Arapura observes that one needs to consider the regularity of the structure sheaf in order to compute the regularity of tensor products. 
	
	For any projective variety $X$ endowed with a globally generated ample line bundle $B$ we fix $$M:=\max\left\{1,\reg_B(\O_X)\right\}$$ for the rest of the section.
	
	\begin{rmk}
		Normal polarized varieties $(X,B)$ of dimension $n\ge2$ with at worst $\mathbf{Q}$-factorial terminal singularities having $M=1$ are classified. Indeed, $\reg(\O_X)\le1$ implies \[
		h^0(X,K_X+(n-1)B)=h^n(X,(1-n)B)=0\ \ \ \mbox{and}\ \ \ h^1(X,\O_X)=0
		\] which, by \cite[Corollary 7.28 \& Table 7.1]{beltrametti2011adjunction}, force $(X,B)$ to be one of the following: 
		\begin{itemize}
			\item $(\P^n,\OPn(1))$;
			\item $(Q^n,\O_{\P^{n+1}}(1)_{|Q^n})),$ where $Q^n\subset\P^{n+1}$ is a quadric;
			\item $(\P(\F),\O_{\P(\F)}(1)),$ where $\F$ is an ample vector bundle of rank $n$ over $\P^1$;
			\item a possibly degenerate generalized cone $C_n(\P^2,\O_{\P^2}(2))$ over $(\P^2,\O_{\P^2}).$
		\end{itemize}
	\end{rmk}
	
	The following is a version of \cite[Example 1.8.7 \& Proposition 1.8.8]{lazarsfeld2017positivity} for every polarized variety. This is proven in \cite[Corollary 3.2 \& Lemma 3.9]{arapura2004frobenius}.
	
	\begin{lemma}\label{lem:m-regular-resolution}
		Let $X$ be a projective variety of dimension $n\ge1$ and let $B$ be a globally generated ample line bundle on $X.$ 
		
		A $p$-regular coherent sheaf $\F$ on $X$ admits a long resolution
		\begin{equation}\label{eq:resolution}
			F_\bullet\colon\hspace{1cm}	\cdots\longrightarrow W_\l\o B^{\otimes(-p-\l M)}\longrightarrow\cdots\longrightarrow W_1\o B^{\o(-p-M)}\longrightarrow W_0\o B^{\o(-p)}\longrightarrow\F\longrightarrow0
		\end{equation}
		where the $W_i'$s are some finite-dimensional vector spaces and $W_0=H^0(X,\F(pB)).$ 
		
		Conversely, if a coherent sheaf $\g$ on $X$ admits a possibly infinite resolution by coherent sheaves \[
		G_\bullet\colon\hspace{1cm}\cdots\longrightarrow G_h\longrightarrow\cdots\longrightarrow G_1\longrightarrow G_0\longrightarrow\g\longrightarrow 0
		\] with $G_j$ being $q_j$-regular for $0\le j\le n-1$ (resp. $0\le j\le n$), then $\g$ is $q$-regular (resp. the map  $$H^0(X,G_0(q'B))\to H^0(X,\g(q'B))$$ is surjective), with $q=\max_{0\le j\le n-1}\{q_j-j\}$  (resp. $q'=\max_{0\le j\le n}\{q_j-j\}$).
	\end{lemma}
	
	This leads to a generalization of \cite[Proposition 1.8.9]{lazarsfeld2017positivity}.
	
	\begin{cor}\label{cor:regularity-tensor}
		Let $X$ be a projective variety of dimension $n\ge1$ and let $B$ be a globally generated ample line bundle on $X.$ Let $\E$ and $\F$ be coherent sheaves on $X$ such that at every point of $X$ either $\E$ or $\F$ is locally free. If $\E$ is $e$-regular and $\F$ is $f$-regular, then $\E\o\F$ is $\left(e+f+(n-1)(M-1)\right)$-regular and the multiplication map $$H^0(X,\F(fB))\o H^0(\E((e+n(M-1))B))\rightarrow  H^0(X,(\E\o\F)(e+f+n(M-1))B)$$ is surjective.
	\end{cor}
	\begin{proof}
		Consider the resolution $F_\bullet$ of $\F$ given in (\ref{eq:resolution}) and twist it through by $\E.$ The resulting complex
		\[
		\cdots\longrightarrow {W_\l\o\E((-f-\l M)B)}\longrightarrow\cdots\longrightarrow {W_1\o\E((-f-M)B)}\longrightarrow W_0\o\E(-fB)\xlongrightarrow{\eps}{\E\o\F}\longrightarrow0
		\] is still exact: as $W_\l\o B^{\o(-f-\l M)}$ is flat, the claim follows by the fact that it remains exact on the right since at stalk level either $\E$ or $\F$ is flat. We immediately see that $W_j\o\E((-f-jM)B)$ is $(e+f+jM)$-regular for every $0\le j\le n.$ Indeed, for $i>0$ we have 
		\[
		H^i(X,W_j\o\E((-f-jM)B)(e+f+jM-i)B)\cong H^i(X,\E((e-i)B))^{\oplus\dim W_i}=0
		\] by the $e$-regularity of $\E.$ The conclusion follows by the second part of Lemma \ref{lem:m-regular-resolution}.
	\end{proof}
	
	\begin{lemma}\label{lem:regularity-tensor}
		Let $X$ be a projective variety and let $B$ be a globally generated ample line bundle on $X.$  Let $\E$ and $\F$ be coherent sheaves on $X.$ If $\E$ is $e$-regular and $\F$ is $f$-regular, then $\E\o\F$ is $(e+f)$-regular assuming one of the following holds:
		\begin{itemize}
			\item[(1)] $M=1$ and at every point of $X$ either $\E$ or $\F$ is locally free.
			\item[(2)] $X$ is a curve and at every point of $X$ either $\E$ or $\F$ is locally free.
			\item[(3)] $X$ is a surface and $B$ is very ample.
		\end{itemize} 
	\end{lemma}
	
	\begin{proof}
		If either (1) or (2) holds, the claim follows immediately from Corollary \ref{cor:regularity-tensor}. Finally (3) is just \cite[Lemma 1.4 \& Proposition 1.5]{sidman2002castelnuovomumford}.
	\end{proof}
	
	\begin{cor}\label{cor:regular-symmetric}
		Let $X$ be a projective variety of dimension $n\ge1$ and let $B$ be a globally generated ample line bundle on $X.$ Given an $e$-regular vector bundle $\E$ on $X,$ the bundles $\E^{\o p},\S^p\E,\Lambda^p\E$ are $(pe)$-regular if one of the following holds:
		\begin{itemize}
			\item[(1)] $M=1.$
			\item[(2)] $X$ is a curve.
			\item[(3)] $X$ is a surface and $B$ is very ample.
		\end{itemize}
	\end{cor}
	\begin{proof}
		Under one of those assumptions, the $p$-fold tensor power $\E^{\o p}$ is $(pe)$-regular by Lemma \ref{lem:regularity-tensor}. As we are working over the field of complex numbers, the $p$-th symmetric power $\S^p\E$ and $p$-th exterior power $\Lambda^p\E$ are direct summands of $\E^{\o p}.$ Hence they must be $(pe)$-regular as well.
	\end{proof}
	
	Another property of the polarization which allows the regularity to behave well under the tensor operation is the  \emph{Koszul} property. We refer to \cite{froberg2023koszul,totaro2013line, conca2013koszul} for more details.

	\begin{defi}
		A line bundle $L$ on a $n$-dimensional projective variety $X$ is \emph{$K$-Koszul} if $L$ is very ample and its section ring $R(L)=\bigoplus_mH^0(X,mL)$ determines a resolution of $\C$ as graded $R(L)$-module
		\[
		\cdots\longrightarrow M_K\longrightarrow\cdots\longrightarrow M_1\longrightarrow M_0\longrightarrow \C\longrightarrow0
		\] with $M_i=\bigoplus R(L)(-i)$ being a free $R(L)$-module generated in degree $i$ for every $i\le K.$ We say that $L$ is \emph{Koszul ample} if $L$ is $2n$-Koszul. We say that $L$ is \emph{Koszul} if it is $K$-Koszul for all $K.$
	\end{defi}
	
	\begin{ex}\label{ex:koszul}
		Examples of embedded $n$-dimensional smooth projective varieties $X\subset\P^N$ whose hyperplane section $H=\O_X(1)$ is Koszul are: smooth complete intersections of type $(2,2,\dots,2)$ \cite[§3.1]{froberg2023koszul}, anticanonically embedded Del Pezzo surfaces of degree $d\ge4$ \cite[Remark 3.22]{bangere2024koszul}, canonically embedded curves which are neither hyperelliptic nor trigonal nor plane quintics \cite{pareschi1997canonical}, embedded homogeneous varieties $X=G/P$ with $G$ being a simply connected semisimple algebraic group and $P$ a parabolic subgroup of $G$ \cite{ravi1995coordinate}, abelian varieties such that $H=L^{\o p}$ for some ample line bundle $L$ on $X$ and $p\ge4$ \cite[Theorem 1]{kempf1989projective}, embedded varieties such that $H=B^{\o k}$ for any $k>\reg_B(B)$ with $B$ very ample on $X$ \cite[Theorem 3.3]{hanumanthu2010some}, embedded varieties with trivial canonical bundle such that $H=A^{\o (n+1)}$ for a very ample line bundle $A$ on $X$ \cite[Theorem B]{pareschi1993koszul}.
	\end{ex}
	
	\begin{prop}\label{prop:multiplication-map}
		Let $\E$ and $\F$ be vector bundles on $X$ that are respectively $e$-regular and $f$-regular. Then:
		\begin{itemize}
			\item[(1)] $\E\o\F$ is $(e+f)$-regular if $B$ is Koszul ample.
			\item[(2)] $H^0(X,\E(eB))\o H^0(X,\F(fB))\to H^0(X,(\E\o\F)((e+f)B))$ is surjective if $B$ is $3n$-Koszul.
		\end{itemize}
	\end{prop}
	\begin{proof} See \cite[Theorem 6.8 \& Lemma 6.9 (First Version)]{raychaudhury2022positivity}.
	\end{proof}
	
	\section{Projective normality of vector bundles}
	After having recalled the definition of projective normality property for line bundles, we extend these notions to higher rank vector bundles in the obvious way. 
	
	\begin{defi} 
		A line bundle $L$ on a projective variety $X$ is said $k$-normal for some $k\ge1$ if the natural map $$\S^kH^0(X,L)\to H^0(X,kL)$$ is surjective. We say $L$ is \emph{normally generated} if it is ample and $k$-normal for all $k\ge1.$
		
		An embedded projective variety $X\subset\P^N$ is said $m$-normal for some $m\ge1$ if the restriction map $$H^0(\P^N,\OPN(m))\to H^0(X,\O_X(m))$$ is surjective, or equivalently if $H^1(\P^N,\I_{X/\P^N}(m))=0.$ A $1$-normal variety is said \emph{linearly normal.} The variety $X\subset\P^N$ is \emph{projectively normal} if it is $m$-normal for all $m\ge1.$
	\end{defi}
	
	The following observation is very well-known.
	
	\begin{rmk}\label{rmk:proj-norm}
		Let $L$ be an ample line bundle on a projective variety $X.$ Then the following are equivalent:
		\begin{enumerate}[(i)]
			\item $L$ is normally generated.
			\item $H^0(X,L)^{\o k}\to H^0(X,kL)$ is surjective for all $k\ge1.$
			\item $H^0(X,L)\o H^0(X,kL)\to H^0(X,(k+1)L)$ is surjective for all $k\ge 0.$
			\item $L$ is very ample and embeds $X\hookrightarrow\P(H^0(X,L))$ as a projectively normal variety. 
		\end{enumerate}
	\end{rmk}
	
	Therefore a normally generated line bundle $L$ on a projective variety $X$ induces an embedding $\phi_L\colon X\hookrightarrow\P^N$ as a projectively normal scheme. In the following we are going to identify normal generation (of $L$) and projective normality (of $\phi_L(X)\subset\P^N$). Moreover in this case the coordinate ring $R_X$ and the section ring $R(L)$ coincides.
	
	The natural generalization of these notions to vector bundles is asking them to hold for the tautological bundle.
	
	\begin{defi}
		We say that a vector bundle $\E$ on a projective variety is \emph{projectively normal} (resp. \emph{$k$-normal} for some $k\ge1$) if $\OPE(1)$ is normally generated (resp. $k$-normal) on $\P(\E).$ 
	\end{defi}
	
	\begin{rmk}\label{rmk:cohomology-serre}
		Given a vector bundle $\E$ of rank $r$ on a projective variety $X,$ by \cite[Exercises III.8.1-III.8.4]{hartshorne2013algebraic}  there are isomorphisms $$H^i(\P(\E),\O_{\P(\E)}(k))\cong H^i(X,\S^k\E)$$ for every $i,k\ge 0.$ In particular, by Remark \ref{rmk:proj-norm}, $\E$ is $k$-normal for some $k\ge1$ (resp. projectively normal) if and only if $$\S^kH^0(X,\E)\to H^0(X,\S^{k}\E)\hspace{0.2cm} \mbox{or, equivalently,}\hspace{0.2cm} H^0(X,\E)^{\o k}\to H^0(X,\S^k\E)$$ is surjective (resp. is surjective for all $k\ge1$ and $\E$ is ample). 
		
		A stronger condition to get $k$-normality (resp. projective normality) of $\E$ is requiring that multiplication the map
		\[
		\mu_\E^k\colon H^0(X,\E)^{\o k}\to H^0(X,\E^{\o k})
		\] is surjective (resp. is surjective for all $k\ge1$ and $\E$ is ample).  Indeed, one has the commutative diagram
		\[\begin{tikzcd}
			&& {H^0(X,\E)^{\o k}} \\
			{H^0(X,\E^{\o k})} &&&& {H^0(X,\S^{k}\E)}
			\arrow[from=1-3, to=2-1, "\mu_\E^k"']
			\arrow[from=1-3, to=2-5]
			\arrow[from=2-1, to=2-5]
		\end{tikzcd}\]  for every $k\ge1.$ As the $h$-symmetric product $\S^h\E$ is a direct summand of the $h$-tensor power $\E^{\o h}$ for all $h\ge0,$ the horizontal map is always surjective. Thus the conclusion immediately follows. 
	\end{rmk}
	
	We are going to use this remark with no further mention.
	
	\begin{defi}
		A vector bundle $\E$ on a projective variety $X$ is  \emph{strongly $k$-normal} if 
		\[
		\mu_\E^k\colon H^0(X,\E)^{\o k}\to H^0(X,\E^{\o k})
		\] is surjective. When $k=2,$ we will write $\mu_\E=\mu_\E^2.$
	\end{defi}
	
	\begin{rmk}\label{rmk:direct-sum}
		Let $\E\cong\F^{\oplus s}$ for a vector bundle $\F$ on a projective variety $X$ and for some $s\ge1.$ Then $\E$ is strongly $2$-normal if and only if $\F$ is strongly $2$-normal.
		
		This follows by observing that $$\mu_\E\colon H^0(X,\E)\o H^0(X,\E)\cong \left(H^0(X,\F)\o H^0(X,\F)\right)^{\oplus s^2}\to H^0(X,\F\o \F)^{\oplus s^2}\cong H^0(X,\E\o\E)$$	is surjective if and only if each $\mu_\F\colon H^0(X,\F)\o H^0(X,\F)\to H^0(X,\F\o\F)$ is surjective.
	\end{rmk}
	
	\begin{prop}\label{prop:proj-norm-0-reg}
		Let $X\subset\P^N$ be a projective variety of dimension $n\ge1$ such that $\O_X(1)$ is $3n$-Koszul. Then every $0$-regular vector bundle on $X$ is strongly $k$-normal for all $k\ge1.$ In particular, all ample $0$-regular vector bundles are projectively normal.
	\end{prop}
	
	This Proposition applies for instance to all embedded varieties in Example \ref{ex:koszul}.
	
	\begin{proof}
		Let $\E$ be a $0$-regular vector bundle on $X\subset\P^N$ and proceed by induction on $k\ge1.$ As the base case $k=1$ is trivial, suppose $k>1$ and consider the commutative diagram 
		\[\begin{tikzcd}
			& {H^0(X,\E)^{\o k}} \\
			\\
			{H^0(X,\E)\o H^0(X,\E^{\o(k-1)})} &&& {H^0(X,\E^{\o k})}.
			\arrow["{\mr{id}\o\mu_{\E}^{k-1}}"', from=1-2, to=3-1]
			\arrow["{\mu_\E^k}", from=1-2, to=3-4]
			\arrow["{m^k_\E}"', from=3-1, to=3-4]
		\end{tikzcd}\]		
		The map $\mr{id}\o\mu_\E^{k-1}$ is surjective by the inductive hypothesis and $m_\E^k$ is surjective by Proposition \ref{prop:multiplication-map}. Thus $\mu_\E^k=m_\E^k\circ(\mr{id}\o\mu_\E^{k-1})$ is surjective as required.
	\end{proof}
	
	Strongly $2$-normality on $0$-regular vector bundles on low dimensional varieties directly implies projective normality.
	
	\begin{rmk}\label{rmk:strong-k-normality-curve} 
		Let $(X,B)$ be either an irreducible curve with a globally generated ample line bundle or an irreducible surface with a very ample line bundle.	Then a strongly $2$-normal $0$-regular vector bundle for $(X,B)$ is automatically strongly $k$-normal for all $k\ge2.$ In particular, all ample $0$-regular and strongly $2$-normal vector bundles for $(X,B)$ are projectively normal.
		
		Indeed, let $\E$ be a strongly $2$-normal $0$-regular vector bundle for  $(X,B)$ and let $M_\E$ be its syzygy bundle. Tensoring the syzygy exact sequence of $\E$ through by $\E,$ we immediately see that the strongly $2$-normality implies $H^1(X,M_\E\o\E)=0.$ If $\dim X=1,$ this amounts to say that $M_\E\o\E$ is $1$-regular. If $X$ is a surface, observing that also $\E\o\E$ is $0$-regular (Corollary \ref{cor:regular-symmetric}), the exact sequence
		\[
		0=H^1(X,\E\o\E(-B))\longrightarrow H^2(X,M_\E\o\E(-B))\longrightarrow H^0(X,\E)\o H^2(X,\E(-B))=0
		\]
		tells that $M_\E\o\E$ is $1$-regular also in this case. Then, by Lemma \ref{lem:regularity-tensor}, $M_\E\o\E^{\o j}$ and $\E^{\o j}$ are respectively $1$-regular and $0$-regular for all $j\ge 1.$ Proceeding inductively on $k\ge 2,$ suppose $k>2$ and assume that $\mu_\E^h$  is surjective for all $1\le h\le k-1.$ Using the $1$-regularity of $M_\E\o\E^{\o(k-1)}$ we immediately get the surjectivity of $$H^0(X,\E)\o H^0(X,\E^{\o (k-1)})\to H^0(X,\E^{\o k}).$$ Since $H^0(X,\E)^{\o k}=H^0(X,\E)\o H^0(X,\E)^{\o(k-1)}\to H^0(X,\E)\o H^0(X,\E^
		{\o(k-1)})$ is onto by the inductive hypothesis, the assertion follows. 
	\end{rmk}
	
	The following result is well-know, see for instance \cite[Fact 1.7]{miro1994cohomological} or \cite[p. 1014]{tripathi2016splitting}. For more details we refer to \cite{buchsbaum1975generic,eisenbud2013commutative, weyman2003cohomology}. 
	
	\begin{lemma}\label{lem:fact}
		Let $0\to\E\to\F\to\g\to0$ be an exact sequence of vector bundles on a variety $X.$ Then we have the following exact sequences of vector bundles for every $k\ge1$:
		\begin{align}
			0&\to \L^k\E\to\L^k\F\to\L^{k-1}\F\o\g\to\cdots\to\F\o\S^{k-1}\g\to\S^{k}\g\to0, \label{eq:fact1}\\
			0&\to\S^k\E\to\S^{k-1}\E\o\F\to\cdots\to\E\o\L^{k-1}\F\to\L^k\F\to\L^k\g\to0, \label{eq:fact3}
		\end{align}
	\end{lemma}
	
	This lemma shows that we can obtain the $k$-normality through some cohomology vanishings.
	
	\begin{ex}\label{ex:2-normality}
		Let $X$ be a smooth regular projective variety and let $\E$ be a globally generated vector bundle on $X$ such that $H^1(X,\E)=0.$ Then $\E$ is $2$-normal if $H^2(X,\L^2M_\E)=0.$ The converse holds if $H^2(X,\O_X)=0.$
		
		To see this, consider (\ref{eq:fact1}) with $k=2$ for the syzygy exact sequence of $\E$ and then split it in the two exact sequences
		\[
		0\to K\to H^0(X,\E)\o\E\to\S^2\E\to0\ \text{and}\ 0\to\L^2M_\E\to\L^2H^0(X,\E)\o\O_X\to K\to0.
		\] Taking the cohomology of the first one, we immediately see that $H^0(X,\E)\o\ H^0(X,\E)\to H^0(X,\S^2\E)$ is surjective if and only if $H^1(X,K)=0.$ The cohomology of the second one yields the exact sequence $$0=\L^2H^0(X,\E)\o H^1(X,\O_X)\to H^1(X,K)\to H^2(X,\L^2M_\E)\to \L^2H^0(X,\E)\o H^2(X,\O_X).$$ The claim is now obvious.
	\end{ex}
	
	\begin{notat}\label{not:non-special-vector-bundle}
		We say that a coherent sheaf on a variety is \emph{non-special} if its first cohomology group is $0.$ It is \emph{special} otherwise.
	\end{notat}
	
	We observe that projective normality is an open property in proper flat families of non-special vector bundles. This will allow us to consider open subsets of projectively normal Ulrich bundles in the moduli space of vector bundles.
	
	\begin{lemma}\label{lem:open-condition}
		Let $f\colon Y\to S$ be a proper morphism over a Noetherian scheme $S$ and let $\F$ be a coherent sheaf on $Y$ that is flat over $S.$ Suppose there is $0\in S$ such that $\F_0=\F_{|Y_0}$ is locally free, semistable, non-special, ample, globally generated and $2$-normal (resp. strongly $2$-normal). Then there exists an open neighborhood $U\subset S$ of $0$ such that $\F_s=\F_{|Y_s}$ satisfies all the above properties for every $s\in U.$ 
	\end{lemma}
	
	\begin{proof}
		The assertion is local on the target, so we assume $S$ is affine. Thanks to \cite[Lemma 2.1.8 \& Proposition 2.3.1]{huybrechts2010geometry} and to \cite[Corollary 23.144]{gortz2023algebraic}, up to shrink $S,$ 
		we can suppose that $\F_s$ is locally free, semistable and non-special for every $s\in S.$ In particular, $\F$ is locally free on $Y$ by \cite[Lemma 2.1.7]{huybrechts2010geometry}. Shrinking again, we can assume also  $f_\ast\F$ is locally free as well \cite[Corollary 23.144]{gortz2023algebraic}. Up to replace $S$ with a suitable open neighborhood of $0,$ by \cite[Proposition 6.1.9]{lazarsfeld2017positivity2} we can  suppose $\F_s$ is ample for every $s\in S.$ Note that the function $s\mapsto h^1(Y_s,\F_s)=0$ is constant. 
		
		Letting $\rho\colon f^\ast f_\ast\F\to\F$ be the natural morphism, we can see that $\F_s$ is not globally generated if and only if $\rho_s=\rho_{|Y_s}\colon H^0(Y,\F)\o\O_{Y_s}\to \F_s$ is not surjective. To prove this, observe that there is a factorization $$\rho_s=e_s\circ (r_s\o\mr{id_{\O_{Y_s}}})$$ where $e_s\colon H^0(Y_s,\F_s)\o\O_{Y_s}\to\F_s$ is the evaluation map and $r_s\colon H^0(Y,\F)\to H^0(Y_s,\F_s)$ is the restriction map. If we prove that $r_s$ is surjective, then the claim will be clear. In order to do this,  observe that, since $h^1(Y_s,\F_s)=0$ for all $s,$ \cite[Corollary 23.144]{gortz2023algebraic} provide the isomorphism $$(f_\ast\F)_{|\Spec(\C(s))}\cong (f_{|Y_s})_\ast(\F_s)$$ for every $s\in S.$ Since $S$ is affine, thanks to \cite[Proposition III.8.5 \& Proposition II.5.2(e)]{hartshorne2013algebraic} this yields  
		\begin{align*}
			\w{H^0(Y,\F)\o \C(s)}&\cong \w{H^0(S,f_\ast\F)\o \C(s)}\\
			&\cong (\w{H^0(S,f_\ast\F)})_{|\Spec(\C(s))}\\
			& \cong \w{H^0(\Spec(\C(s)),(f_{|Y_s})_\ast(\F_s))}\\
			& \cong \w{H^0(Y_s,\F_s)}.
		\end{align*}
		In virtue of the equivalence of categories between $\O_S$-modules over $S$ and $\G(S,\O_S)$-modules, we thus obtain the isomorphism $$H^0(Y,\F)\o \C(s)\cong H^0(Y_s,\F_s).$$ Since $S$ is affine, the locally free sheaf $f_\ast\F$ is globally generated. Therefore the restriction of the surjective evaluation map $H^0(S,f_\ast\F)\o\O_S\to f_\ast\F$ to $\Spec(\C(s))$ yields the surjective map
		\begin{equation*}
			r_s\colon H^0(Y,\F)\cong H^0(S,f_\ast\F)\to H^0(\Spec(\C(s)),(f_\ast\F)_{|\Spec(\C(s))}) \cong H^0(S,(f_{|Y_s})_\ast(\F_s))\cong H^0(Y_s,\F_s),
		\end{equation*} as desired.
		
		As $\F_0$ is generated by global sections, the coherence of $\mr{coker}(\rho)$ tells that there is an affine open neighborhood of $0$ such that $\rho$ is surjective on every fibre $Y_s.$ Hence, up to shrink $S,$ we suppose that $\F_s$ is globally generated for all $s,$ so that $\rho$ is globally surjective. 
		
		Now, $\P(\F_s)$ is a fibre of the composite proper morphism $\P(\F)\to Y\to  S$ and $\mc{L}=\O_{\P(\F)}(1)$ is a line bundle on $\P(\F)$ such that $\mc{L}_s=\O_{\P(\F_s)}(1)$ for every $s\in S.$ The surjectivity of the multiplication map $$H^0(\P(\F_0),\mc{L}_0)\o H^0(\P(\F_0),\mc{L}_0)\to H^0(\P(\F_0),\mc{L}_0\o\mc{L}_0)$$ is well-known to be an open condition by semicontinuity (see \cite[Proof of Lemma 1.3, lines 7-8]{martens1985normal}), thus the conclusion in case of the $2$-normality of $\F_0$ follows.
		
		As a recap, we have that $\F_s$ is locally free, semistable, non-special, ample and globally generated for all $s\in S.$ Assuming that $\F_0$ is strongly $2$-normal, we need to prove that there is a neighborhood $U\subset S$ of $0$ such that $\F_s$ is strongly $2$-normal for all $s\in U.$ To do this, consider the morphism $\mu\colon H^0(Y,\F)\o\F\to\F\o\F.$ We claim that $\mu$ is surjective on $Y.$ Indeed, as $f_\ast\F$ is globally generated, then so is $f^\ast f_\ast\F.$ Thanks to the surjectivity of $\rho,$ we deduce that $\F$ is globally generated as well. The assertion follows by tensoring the syzygy exact sequence of $\F$ through by $\F$ itself.  
		In conclusion, we have the exact sequence 
		\[
		\begin{tikzcd}
			0\rar&\K=M_\F\o\F\rar& H^0(Y,\F)\o\F\rar&\F\o\F\rar&0.
		\end{tikzcd}
		\]   
		Note also that $M_\F,$ hence $\K,$ is flat over $S$ by \cite[Exercise 2.25]{atiyah2018introduction}. 
		Hence the above sequence restricts to the short exact sequence
		\[
		\begin{tikzcd}
			0\rar&\K_s\rar& H^0(Y,\F)\o\F_s\arrow[rr, "\mu_s=\mu_{|Y_s}"]&&\F_s\o\F_s\rar&0
		\end{tikzcd}
		\] for each $s\in S.$  
		Taking the cohomology, we obtain the following commutative diagram {\small $$\begin{tikzcd}
				H^0(Y,\F)\o H^0(Y_s,\F_s) && H^0(Y_s,\F_s\o\F_s) & H^1(Y_s,\K_s) & 0 \\
				& H^0(Y_s,\F_s)\o H^0(Y_s,\F_s)
				\arrow[from=1-1, to=1-3, "\mu_s(Y_s)"]
				\arrow[from=1-3, to=1-4]
				\arrow[from=1-1, to=2-2, "r_s\o\mr{id}"']
				\arrow[from=2-2, to=1-3, "\mu_{\F_s}"']
				\arrow[from=1-4, to=1-5]
			\end{tikzcd}$$} 		where the row is exact. Since $r_s$ is surjective, we deduce that $\mu_s(Y_s)$ is surjective if and only if $\mu_{\F_s}$ is surjective. Now, $\F_0$ is strongly $2$-normal, and this forces $\mu_{\F_0},$ hence $\mu_0(Y_0),$ to be onto. The exactness of the cohomology sequence yields $H^1(Y_0,\K_0)=0.$ As $\K$ is flat over $S,$  the semicontinuity theorem applied to $s\mapsto h^1(Y_s,\K_s)$ provides an open neighborhood $U\subset S$ of $0$ where $h^1(Y_s,\K_s)=0$ for all $s\in U.$  We obtain the surjectivity of $\mu_s(Y_s),$ thus of $\mu_{\F_s},$ for all $s\in U$ as required.
	\end{proof}
	
	It's well known that a normal projective variety $X\subset\P^N$ is projectively normal if and only if $R_X$ is integrally closed (see for instance \cite[Exercise II.5.14]{hartshorne2013algebraic}). A related property on $R_X$ is being \emph{Cohen-Macaulay.}
	
	\begin{defi}
		An embedded projective variety $X\subset\P^N$ is said \emph{arithmetically Cohen-Macaulay} (\emph{aCM} for short) if the homogeneous coordinate ring $R_X$ is Cohen-Macaulay, i.e. $\dim R_X=\depth R_X.$ For a vector bundle $\E$ on a projective variety $Y,$ we say that $\P(\E)$ is \emph{aCM} if $\OPE(1)$ induces an embedding $\P(\E)\subset\P(H^0(Y,\E))$ which is aCM.
	\end{defi}
	
	As is well-known, aCM property is stronger than projective normality.
	
	\begin{rmk}\label{rmk:acm-variety}
		A projective variety $X\subset\P^N$ of positive dimension is aCM if and only if $H^1(\P^N,\I_{X/\P^N}(k))=0$ for all $k\ge0$ and $\O_X$  is an aCM sheaf with respect to $\O_X(1)$ if and only if $H^i(\P^N,\I_{X/\P^N}(k))=0$ for $1\le i\le \dim X$ and $k\in\Z$ \cite[Proposition 2.1.9]{costa2021ulrich}. In particular, an irreducible projective curve $C\subset\P^N$ is aCM if and only if it is projectively normal. 
	\end{rmk}
	
	Let's see how $k$-normality and aCM property behave with respect to hyperplane sections.
	
	\begin{lemma}\label{lem:proj-norm-hyperplane-section}
		Let $X\subset\P^N$ be a linearly normal projective variety of dimension $n\ge 2$ and let $Y= X\cap\P^{N-1}$ be a linearly normal irreducible hyperplane section. Then the following holds:
		\begin{itemize}
			\item[(i)] For $k>0,$ if the embedding $Y\subset\P^ {N-1}$ is $h$-normal for $2\le h\le k,$ then so is the embedding $X\subset\P^N.$ Conversely,  $Y\subset\P^{N-1}$ is $k$-normal for $k\ge 0$ if $X\subset\P^N$ is $k$-normal and $H^1(X,\O_X(k-1))=0.$
			\item[(ii)] If $X$ is aCM, then so is $Y.$ The converse holds if $X$ is locally Cohen-Macaulay.
		\end{itemize}
	\end{lemma}
	\begin{proof} Let's consider (i). 		For every $k\ge 0$ we have the following commutative diagram with exact rows 
		{\small \[
			\begin{tikzcd}
				0 & H^0(\P^N,\OPN(k-1)) & H^0(\P^N,\OPN(k)) & H^0(\P^{N-1},\O_{\P^{N-1}}(k)) & 0 \\
				\\
				0 & H^0(X,\O_X(k-1)) & H^0(X,\O_X(k))& H^0(Y,\O_Y(k)) & H^1(X,\O_X(k-1)).
				\arrow[from=1-1, to=1-2]
				\arrow[from=1-2, to=1-3]
				\arrow[from=1-3, to=1-4]
				\arrow[from=1-4, to=1-5]
				\arrow[from=3-1, to=3-2]
				\arrow[from=3-2, to=3-3]
				\arrow[from=3-3, to=3-4, "\rho_k"]
				\arrow[from=1-2, to=3-2, "r_{k-1}"]
				\arrow[from=1-3, to=3-3, "r_k"]
				\arrow[from=1-4, to=3-4, "r'_k"]
				\arrow[from=3-4, to=3-5]
			\end{tikzcd}
			\]}    For the first part, we proceed by induction on $h,$ with $2\le h\le k.$ The base case $h=2$ is obtained immediately from the Snake lemma since $r_1$ is surjective, by the $1$-normality of $X,$ as well as $r'_2$ by hypothesis.  The inductive step with $2<h\le k$ follows again by the Snake lemma, the inductive hypothesis and the surjectivity of $r_h'.$\\
		For the converse, our assumption implies the surjectivity of the composite map $\rho_k\circ r_k.$ The commutativity of the right square tells that $r'_k$ is onto as desired.
		
		To prove the first part of (ii), consider the short exact sequence \[
		\begin{tikzcd}
			0\rar&\O_X(j-1)\rar&\O_X(j)\rar&\O_Y(j)\rar&0
		\end{tikzcd}
		\] for $j\in\Z.$ Remark \ref{rmk:acm-variety} and (i) imply that $Y\subset\P^{N-1}$ is projectively normal. Moreover, it is clear from the cohomology of the above sequence that $H^i(Y,\O_Y(j))=0$ for all $j\in\Z$ and $0<i<n-1.$ Therefore $Y\subset\P^{N-1}$ is aCM by Remark \ref{rmk:acm-variety}. For the converse, see \cite[Theorem 1.3.3]{migliore1998introduction}.
	\end{proof}
	
	\begin{rmk}\label{rmk:sectional-curve-vector-bundle}
		Let $X$ be a regular smooth projective variety of dimension $n\ge1$ and let $\E$ be a very ample vector bundle of rank $r\ge2$ on $X.$ Then $\P(\E)$ is aCM as soon as $$(3-n)s_n(\E^\ast)\ge  3+(K_X+\det(\E))\cdot s_{n-1}(\E^\ast)$$ where $s_k(\E^\ast)$ is the $k$-th Segre class of $\E^\ast.$\footnote{For the definition and the properties of the Segre classes of a vector bundle $\E$ we refer to \cite[§10.1]{eisenbud20163264}, where the authors use the convention $\P(\E)=\mr{Proj}(\Sym(\E^\ast)).$}
		
		To see this, consider a smooth sectional curve $C\subset\P(\E)\subset\P(H^0(X,\E))=\P^M$ for the embedding defined by $|\OPE(1)|$ and let $g=g(C)$ be its genus. Letting $\xi$ be the class of $\OPE(1),$ we know that $$\deg C=\deg\P(E)=\xi^{n+r-1}=s_n(\E^\ast).$$ Observe that this is linearly normal since $H^1(\P(\E),\OPE)=H^1(X,\O_X)=0.$ By Lemma \ref{lem:proj-norm-hyperplane-section} we can reduce to study the projective normality of $C\subset\P^{M+2-n-r}.$ A classical result states that $C\subset\P^{M+2-n-r}$ is projectively normal if $\deg C\ge 2g+1$ (see, e.g., \cite[§2, Corollary to Theorem 6]{mumford1976algebraic}). Using the adjuction formula and recalling that $K_{\P(\E)}=\pi^\ast(K_X+\det(\E))-r\xi,$ we find that
		\begin{align*}
			g&=1+\frac{1}{2}\left(K_{\P(\E)}\cdot\xi^{n+r-2}+(n+r-2)\xi^{n+r-1}\right)\\
			&=1+\frac{1}{2}\left(\pi^\ast(K_X+\det(\E))\cdot\xi^{n+r-2}+(n-2)\xi^{n+r-1}\right)\\
			&=1+\frac{1}{2}\left((K_X+\det(\E))\cdot s_{n-1}(\E^\ast)+(n-2)s_n(\E^\ast)\right).
		\end{align*}
		The conclusion follows by putting all together.
	\end{rmk}
	
	We conclude the section with a result which says that for very ample vector bundles with no first cohomology group over curves and surfaces with $p_g=0,$ the projective normality is equivalent to the $2$-normality.
	
	\begin{lemma}\label{lem:alzati-russo}
		Let $L$ be a very ample line bundle on a smooth projective variety $Y$ of dimension $n\ge1$ and let $Y\subset\P^N$ be the embedding determined by $|L|.$ If $L$ is $2$-normal and $$H^0(Y,K_Y+(n-2)L)=H^1(Y,L)=H^2(Y,\O_Y)=0,$$ then:
		\begin{itemize}
			\item[(i)] $\I_{Y/\P^N}$ is $3$-regular and $I_{Y/\P^N}$ is generated in degree less than or equal to $3,$
			\item[(ii)] $L$ is projectively normal.
		\end{itemize} 
	\end{lemma}
	
	\begin{proof}
		Combining our assumption with Kodaira vanishing, we immediately see that $H^{i+1}(Y,L^{\o (1-i)})=0$ for all $0\le i\le\dim Y-1.$ As $Y\subset \P^N$ is linearly normal as embedded through a complete linear system, the statement follows by \cite[Proposition 1.2]{alzati2002k} and Remark \ref{rmk:proj-norm}.
	\end{proof}
	
	\begin{prop}\label{prop:alzati-russo}
		Let $Y$ be a smooth projective variety of dimension $m\ge1$ with $H^2(Y,\O_Y)=0$ and let $\E$ be a $2$-normal very ample vector bundle on $Y$ such that $H^1(Y,\E)=H^0(Y,(K_Y+\det(\E))\o\S^{m-3}\E)=0.$ Then  $\E$ is projectively normal, $\I_{\P(\E)/\P(H^0(Y,\E))}$ is $3$-regular and $I_{\P(\E)/\P(H^0(Y,\E))}$ is generated in degree less than or equal to $3.$
	\end{prop}
	\begin{proof}
		We just need to verify the hypothesis of Lemma \ref{lem:alzati-russo}. We already know that $\OPE(1)$ is very ample and $2$-normal, and also that $H^1(\P(\E),\OPE(1))=H^2(\P(\E),\OPE)=0.$ Finally, calling $r$ the rank of $\E,$ we have
		\begin{align*}
			H^0(\P(\E),K_{\P(\E)}+\OPE(m+r-3))&= H^0(\P(\E),\pi^\ast(K_Y+\det(\E))+\OPE(m-3))\\
			&\cong H^0(Y,(K_Y+\det(\E))\o\pi_\ast\OPE(m-3))\\
			&\cong H^0(Y,(K_Y+\det(\E))\o\S^{m-3}\E)\\
			&=0
		\end{align*} as required.
	\end{proof}
	
	\section{Projective normality of Ulrich bundles on curves}
	The goal of this section is to determine under which conditions an Ulrich bundle on a smooth projective curve is projectively normal. Thanks to Proposition \ref{prop:alzati-russo} we only need to study the $2$-normality of an Ulrich bundle. In this section, a curve is always smooth and projective.
	
	\begin{rmk}\label{rmk:ulrich-curve}
		Given any globally generated ample line bundle of degree $d$ on a genus $g$ curve $C,$ a line bundle $\mc{L}$ on $C$ is $B$-Ulrich if and only if $\mc{L}=M+B$ for a non-effective line bundle $M$ of degree $g-1.$ In particular there are $B$-Ulrich bundles of any rank on $C.$ 
		
		The proof is analogous to the case of a very ample polarization. Hence we refer to \cite[§4.1]{coskun2017survey} or \cite[(3.3)]{beauville2018introduction}.
	\end{rmk}
	
	Recall that for a vector bundle $\E$ on a curve, it is defined the quantity
	$$\mu^-(\E)=\min\left\{\mu(Q)\ |\ Q\ \text{is a quotient vector bundle of}\ \E\right\},$$ where $\mu(-)=\deg(-)/\rk(-)$ is the slope of a  vector bundle. We always have $\mu(\E)\ge\mu^-(\E)$ with the equality holding if $\E$ is $\mu$-semistable. We refer to \cite{butler1994normal} for more details.
	
	\begin{prop}\label{prop:proj-norm-curve}
		Let $\E$ be a vector bundle on a smooth projective curve $C$ with slope $\mu^-(\E)>2g(C).$ Then $\E$ is strongly $2$-normal, projectively normal, $\I_{\P(\E)/\P(H^0(C,\E))}$ is $3$-regular and $I_{\P(\E)/\P(H^0(X,\E))}$ is generated in degree $\le3$ in the embedding $\P(\E)\subset\P(H^0(X,\E))$ determined by $|\OPE(1)|.$
	\end{prop}
	\begin{proof} The vector bundle $\E$ is very ample with $H^1(C,\E)=0$ by \cite[Lemma 1.12]{butler1994normal} and strongly $2$-normal by \cite[Theorem 2.1]{butler1994normal}. The assertion follows from Proposition \ref{prop:alzati-russo}. 
	\end{proof}
	
	\begin{cor}\label{cor:ulrich-proj-norm-curve}
		Let $C$ be a smooth projective curve and let $\E$ be a $B$-Ulrich bundle of rank $r$ on $C,$ where $B$ is a globally generated ample line bundle of degree $d$ on $C.$ If $d>g(C)+1,$ then $\E$ is strongly $k$-normal for all $k\ge2,$ projectively normal, and $\I_{\P(\E)/\P^{rd-1}}$ is $3$-regular and $I_{\P(\E)/\P^{rd-1}}$ is generated in degree $\le3$ in the embedding $\P(\E)\subset\P^{rd-1}$ determined by $|\OPE(1)|.$
	\end{cor}
	\begin{proof}
		As stated in Remark \ref{rmk:ulrich-facts}, $\E$ is semistable with slope $\mu(\E)>2g(C).$ The claim follows by Proposition \ref{prop:proj-norm-curve} and Remark \ref{rmk:strong-k-normality-curve}.
	\end{proof}
	
	This condition is certainly not necessary for the projective normality of an Ulrich bundle.
	
	\begin{rmk}\label{rmk:ulrich-proj-norm-plane-curve}
		For a smooth plane $C\subset\P^2$ of degree $d\ge 4$ one has $d\le g+1,$ but all Ulrich bundles on $C\subset\P^2$ are very ample \cite[Theorem 1]{lopez2023geometrical} and also strongly $2$-normal by \cite[Example B.1.3]{lazarsfeld2017positivity} applied to the resolution (\ref{eq:ulrich-tensor-res}) in Lemma \ref{lem:ulrich-tensor-res} below. In particular, by Proposition \ref{prop:alzati-russo} every Ulrich bundle $\E$ on $C\subset\P^2$ is projectively normal, $\I_{\P(\E)/\P^{rd-1}}$ is $3$-regular and $I_{\P(\E)/\P^{rd-1}}$ is generated in degree $\le3$ in the embedding $\phi_{\OPE(1)}\colon\P(\E)\subset\P^{rd-1}.$ 
	\end{rmk}
	
	However, the condition in Corollary \ref{cor:ulrich-proj-norm-curve} is sharp some sense if we consider Ulrich bundles defined with respect to globally generated ample line bundles (and possibly non-very ample).
	
	\begin{rmk}\label{rmk:optimality}
		Let $C$ be a hyperelliptic curve of genus $g=3$ and let $\mc{L}$ be any $K_C$-Ulrich line bundle (Remark \ref{rmk:ulrich-curve}). In particular, since $\mc{L}$ cannot be of the form $K_C(D)$ for an effective divisor $D$ of degree $2,$ we know that $\mc{L}$ is very ample. However, $\mc{L}$ cannot be projectively normal by \cite[Corollary 1.4]{green1986projective}. 
	\end{rmk}
	
	In light of this, we will focus on the case of Ulrich bundles defined with respect to a very ample polarization.	
	
	\begin{rmk}
		Ulrich bundles on a smooth rational curve of degree $d\ge2$ are always projectively normal by Corollary  \ref{cor:ulrich-proj-norm-curve}. Since all very ample line bundles on smooth projective curves of genus $g=1,2$ have degree $d\ge g+2,$ the same is true also in this case. 
	\end{rmk}
	
	We now pass to the study of projective normality of the general $B$-Ulrich bundle on a curve $C.$ We first observe that, as in the case of very ample polarization, they define open subsets in the corresponding moduli space of semistable vector bundles on $C.$ First we state a simple remark.
	
	\begin{rmk}\label{rmk:non-general-divisor}
		Let $C$ be a smooth projective curve of genus $g$ and fix a line bundle $A$ of positive degree $a>0.$ The line bundles of the form $A+D_d,$ for $D_d$ being an effective divisor of degree $d\le g-1,$ are contained in a proper closed subset of $\Pic^{a+d}(C).$
	\end{rmk}
	
	\begin{rmk}\label{rmk:open-ulrich}
		Let $C$ be a smooth projective curve of genus $g$ and let $B$ be a globally generated ample line bundle of degree $d$ on $C.$ By semicontinuity $B$-Ulrich bundles of rank $r,$ which are always semistable, define a non-empty open subset in the good moduli space (resp. moduli stack) of semistable vector bundles of rank $r$ and degree $r(d+g-1)$ (see also \cite[p. 8]{casanellas2012stable} and \cite[p.  95]{coskun2017survey}). Moreover, by Remarks \ref{rmk:ulrich-curve}--\ref{rmk:non-general-divisor}, $B$-Ulrich line bundles form a dense open subset in $\Pic^{d+g-1}(C)$ (see also \cite[(4.1.3)]{coskun2017survey}).
	\end{rmk}
	
	We can now prove one of the main result of this section. 
	
	\begin{prop}\label{prop:ulrich-proj-norm}
		Let $C$ be a smooth projective curve of genus $g\ge2$ and let $B$ be a globally generated line bundle of degree $d>1$ such that there exists a linear series $|V|\subset|B|$ which induces a morphism which is étale onto the schematic image.
		
		If $C$ supports a non-special projectively normal line bundle of degree $d+g-1,$ then the general $B$-Ulrich bundle of rank $r$ is projectively normal, for any $r\ge1.$ In particular, the general rank $r$ Ulrich bundle is projectively normal on $C$ if $d\ge g+2-\mr{Cliff}(C).$\footnote{The Clifford index of a line bundle $L$ on $C$ (of genus $g\ge2$) is the quantity 
			\[
			\mr{Cliff}(L)=\deg(L)-2\dim|L|=\deg(L)-2(h^0(C,L)-1).
			\] The \emph{Clifford index of $C$} is defined as
			$
			\mr{Cliff}(C)=\min\left\{\mr{Cliff}(A)\ \left|\ A\in\Pic(C),\ h^i(C,A)\ge2\ \mbox{for}\ i=0,1\right.\right\}.
			$}
	\end{prop}	
	
	For a detailed account of the main properties of the moduli space (resp. moduli stack) of semistable vector bundles on a smooth projective curve we refer to \cite{alper2013good,alper2022projectivity}.
	
	\begin{proof}
		The second part of the statement will follow from the first one combined with \cite[Theorem 1]{green1986projective}. For the first part, observe that all Ulrich bundles on $C$ are ample \cite[Corollary 3]{buttinelli2024positivity}, then fix $r\ge 1$ and let $M^U_C(r,e)\subset M_{C}^{ss}(r,e)$ be the non-empty open subset of $B$-Ulrich bundles (Remark \ref{rmk:open-ulrich}) of rank $r$ in the irreducible good moduli space of semistable vector bundles of rank $r$ and degree $e=r(d+g-1)$ (see \cite[Theorem 3.12]{alper2022projectivity}).  Then Lemma \ref{lem:open-condition} says that the subset $\mc{P}$ of (semistable) non-special, ample, globally generated, strongly $2$-normal vector bundles $\F$ is open in the (irreducible) moduli stack $\mc{M}_{C}^{ss}(r,e)$ of semistable vector bundles of rank $r$ and degree $e=r(d+g-1)$ on $C.$  As the good moduli space $$f\colon \mc{M}_{C}^{ss}(r,e)\to {M}_{C}^{ss}(r,e)$$ is universally closed \cite[Theorem 3.5]{alper2022projectivity}, the subset
		$$P^c=f(\mc{M}_{C}^{ss}(r,e)\backslash \mc{P})\subset {M}_{C}^{ss}(r,e)$$
		is closed and consists of points $[\E]\in {M}_{C}^{ss}(r,e)$ such that there exists a vector bundle $\F\in[\E]$ which either is special either is not ample either is not globally generated or $\mu_\F$ is not surjective. The complement $P=M^{ss}_C(r,e)\backslash P^c$ is then open and can be described as $$\left\{[\E]\in M^{ss}_C(r,e)\ |\ \exists\F\in[\E]\colon \F\ \text{is non-special, ample, globally generated and strongly}\ 2\text{-normal}\right\}.$$ 
		
		Now, if there exists a non-special projectively normal line bundle in $\Pic^{d+g-1}(C),$ by Lemma \ref{lem:open-condition} (or \cite[Lemma 1.3]{martens1985normal}) we can find a dense open subset consisting of non-special projectively normal line bundles of degree $d+g-1.$ In virtue of the openness of Ulrich line bundles in $\Pic^{d+g-1}(C)$ (Remark \ref{rmk:open-ulrich}), there exists a dense open subset of projectively normal $B$-Ulrich line bundles on $C\subset\P^N.$ Let $\mc{L}$ be one of them. Then $\mc{L}^{\oplus r}$ lies in $M_C^U(r,e)\cap P$ by Remark \ref{rmk:direct-sum}. Therefore the conclusion follows by combining the irreducibility of ${M}_{C}^{ss}(r,e)$ and Proposition \ref{prop:alzati-russo}.
	\end{proof}
	
	As an application of this result, we see that the bound in Corollary \ref{cor:ulrich-proj-norm-curve} can be lowered for the general rank Ulrich bundle.
	
	\begin{lemma}\label{lem:kko}
		Let $C\subset\P^N$ be a smooth projective curve of genus $g\ge3$ and degree $d>1.$ Then:
		\begin{itemize}
			\item[(i)] If $d=g,g+1,$ all Ulrich line bundles on $C$ are projectively normal.
			\item[(ii)] If $g\ge g_h,$ where \[
			g_h=\left\{\begin{array}{rcl}
				15 & \mbox{if} & h=2\\
				17 & \mbox{if} & h=3\\
				27 & \mbox{if} & h=4\\
				33 & \mbox{if} & h=5\\
			\end{array}\right.
			\]
			the general non-special line bundle of degree $2g-h$ is projectively normal.
		\end{itemize}
		In particular, the general rank $r$ Ulrich vector bundle is projectively normal if $d=g,g+1$ and if $d=g-h+1$ and $g\ge g_h$ for $h=2,3,4,5.$
	\end{lemma}
	
	Before seeing the proof, we make a couple of simple observations.
	
	\begin{rmk}\label{rmk:hyperelliptic}
		A smooth projective curve $C\subset\P^N$ of genus $g\ge 2$ and degree $d\le g+1$ is neither hyperelliptic nor elliptic-hyperelliptic, i.e. a double cover of an elliptic curve. 
		
		Indeed, very ample line bundles on smooth hyperelliptic curves of genus $g\ge2$ have degree $d>g+1$ by \cite[Theorem 3.1(3)]{park2008complete}, and a very ample line bundle $L$ on an elliptic-hyperelliptic curve must have $h^1(C,L)\le 1$ \cite[(5)]{martens2012remark} which cannot happen if $\deg(L)\le g+1.$
	\end{rmk}
	
	\begin{proof}[Proof of Lemma \ref{lem:kko}]
		By Proposition \ref{prop:ulrich-proj-norm} we only need to prove (i)-(ii). We observe also that $C$ can be neither hyperelliptic nor elliptic-hyperellitpic (Remark \ref{rmk:hyperelliptic}).
		
		Let's show (i). If $d=g+1,$ all Ulrich line bundle have degree $2g$ and so projectively normal by \cite[Corollary 1.4]{green1986projective}. \\
		Assume $d=g$ and let $H\subset C$ be a hyperplane section. We first show that if $d=g=6,$ then $C$ cannot be a plane quintic. Assuming the contrary, write $H=p_1+\cdots+p_6.$ Since there are no smooth curves of genus $6$ and degree $6$ in $\P^3$ \cite[Example IV.6.4.2]{hartshorne2013algebraic}, we must have $$h^0(C,K_C-p_1-\cdots-p_6)=h^1(C,H)=h^0(C,H)-6-1+6=h^0(C,H)-1\ge 4.$$ However, $h^0(C, K_C-q_1-q_2-q_3)=3$ for every $q_1,q_2,q_3\in C$: the effective divisor $D=q_1+q_2+q_3$ must have $h^0(C,D)=1,$ for otherwise we would get a $g^1_3$ on a plane quintic,  hence Riemann-Roch gives the claim. \\
		Now, $C$ is not a plane quintic and we immediately see that $N\ge 3.$ Using Castelnuovo theorem \cite[Theorem IV.6.4]{hartshorne2013algebraic}, we get $d=g\ge8.$ In this case, an Ulrich line bundle $\mc{L}$ has degree $\deg(\mc{L})=2g-1.$  Thus \cite[Corollary 1.6]{green1986projective} tells that $\mc{L}$ fails to be projectively normal if and only if $C$ is trigonal and $\mc{L}=K_C-E+D$ for an effective divisor $D$ of degree $4$ and $E\in W^1_3(C).$ However this cannot happen for otherwise 
		\[
		h^1(C,L)=h^1(C,K_C-H+D-E)=h^0(C,H-D+E)>0.
		\]
		(Indeed, writing $D=\sum_{i=1}^4p_i$ for some points $p_i\in C$ and using that $H$ is very ample with $h^0(C,H)\ge 4$ and $E\in W^1_3(C),$ we have 
		\[
		h^0(C,H-p_1-p_2-p_3)\ge1,\hspace{1cm} h^0(C,E-p_4)\ge 1,
		\] proving the claim.) Therefore $\mc{L}$ is projectively normal even in this case. This proves (i).
		
		For (ii), we are going to use \cite[Theorems 3.1--3.2]{kato1999normal} and \cite[Theorems 3.1--3.2]{akahori2004remarks}: it is proved that a non-special line bundle $L$ of degree $2g-h$ on a curve of genus $g\ge g_h,$ for $h=2,3,4,5,$ is not projectively normal if and only if either $C$ has a certain gonality or is a covering of a certain curve and  $$L=K_C-c A+ D$$ where  $A\in W^j_a= W^j_a(C),$ $D\in W_b=W^0_b(C)$ and with $1\le c\le 5,$ $j=1,2,$ $3\le a\le 9<g,$ $a\ge 2j,$ $2\le b\le 12<g$ such that $2g-2-c\cdot a+b=2g-h$ given in the Theorems in \emph{loc. cit.} This says that non-special non-projectively normal line bundles are contained in the image of the incidence correspondence
		\[
		\begin{tikzcd}
			& I^{j,h}_{a,b}=\left\{\left.(A, L)\in W^j_a\times\Pic^{2g-h}(C)\ \right|\ L=K_C-c A+ D\ \mbox{for some}\ D\in W_b\right\}  \\
			W_a^j && \Pic^{2g-h}(C)
			\arrow[from=1-2, to=2-1, "\pi_1"']
			\arrow[from=1-2, to=2-3, "\pi_2"]
		\end{tikzcd}\
		\]
		through the projection $\pi_2.$ The fibre of $\pi_1$ over each $A\in W^j_a$ can be identified with the image of $W_b$ under the multiplication map $\Pic^b(C)\xrightarrow{\cong}\Pic^{2g-h}(C)$ by $K_C-cA.$ As seen in Remark \ref{rmk:non-general-divisor}, this set is closed of dimension $b.$ Hence, by \cite[Exercise II.3.22(b)]{hartshorne2013algebraic} and \cite[Theorem IV.5.1]{arbarello1985geometry}, we obtain the bound $$\dim \pi_2(I^{j,h}_{a,b})\le \dim I^{j,h}_{a,b}\le\dim W^j_a+\dim W_b\le a-2j-1+b.$$ By \cite[Theorems 3.1--3.2]{kato1999normal} and \cite[Theorems 3.1--3.2]{akahori2004remarks} we have:
		\begin{itemize}
			\item $(j,a,b)=(1,3,6), (1,4,4)$ when $h=2$;
			\item $(j,a,b)=(1,3,8), (1,4,3),(1,5,4)$ when $h=3$;
			\item $(j,a,b)=(1,3,10),(1,4,6), (1,5,3), (1,6,4), (2,8,6)$ when $h=4$;
			\item $(j,a,b)=(1,3,12), (1,4,5), (1,5,2), (1,6,3), (1,7,4), (2,8,5), (2,9,6)$ when $h=5.$ 
		\end{itemize} For each $(h,j,a,b)$ listed above, we see that $\dim\pi_2(I^{j,h}_{a,b})<g_h\le g.$ Therefore we obtain the assertion in (ii).
	\end{proof}
	
	Thanks to the proof of the Maximal Rank Conjecture for non-special curves \cite{ballico1987maximal}, we can find a lower bound for the projective normality of the general Ulrich bundle on a general curve. Moreover this is sharp for Ulrich line bundles.
	
	\begin{prop}\label{prop:ulrich-proj-norm-general}
		On a general curve of genus $g\ge3,$ the general rank $r$ Ulrich bundle defined with respect to a general very ample polarization of degree $$d\ge \frac{3+\sqrt{8g+1}}{2}$$ is projectively normal. Moreover, this bound is sharp for $r=1.$
	\end{prop}
	\begin{proof}
		We reduce to the case $r=1$ thanks to Proposition \ref{prop:ulrich-proj-norm}. By \cite[Theorem]{ballico1987maximal} (and references therein), on a general curve $C$ of genus $g\ge3$ the general non-special very ample line bundle $L$ of degree $u=g+(d-1)$ inducing an embedding $\phi_L\colon X\hookrightarrow\P^{d-1},$ for $d\ge \frac{3+\sqrt{8g+1}}{2}\ge4,$ satisfies the Maximal Rank Conjecture (MRC), i.e. the restriction map
		\[
		r_{L,k}\colon \S^kH^0(C,L)\to H^0(C,kL)
		\] has maximal rank for all $k\ge1.$ Since a general line bundle $H$ of degree $d$ is very ample \cite[Theorem 5.1.2]{eisenbud1983divisors} and $H$-Ulrich line bundles form an open subset in $\Pic^u(C)$ \cite[(4.1.3)]{coskun2017survey}, we infer that the general $H$-Ulrich line bundle $\mc{L}$ satisfies MRC. Now, Lemma \ref{lem:alzati-russo} says that $\mc{L}$ is projectively normal if and only if it is $2$-normal if and only if $r_{\mc{L},2}=r_2$ is surjective. Given that $r_2$ has maximal rank, this holds if and only if $$\dim \S^2H^0(C,\mc{L})=\binom{d+1}{2}\ge h^0(C,2\mc{L})=2u+1-g=2d+g-1,$$ where the term on the right is easily computed via Riemann-Roch theorem. The assertion is now clear.
	\end{proof}
	
	\begin{rmk}
		Observe that the same bound can be obtained using \cite[Theorem 1]{ballico2010normally} (together with Proposition \ref{prop:ulrich-proj-norm}).
	\end{rmk}
	
	We conclude the section with some remarks on the $(N_p)$ property for vector bundles.
	
	\begin{defi}
		Let $X$ be a projective variety and let $L$ be a very ample line bundle defining an embedding $$\phi_L\colon X\to\P(H^0(X,L))=\P^N.$$ Then $R(L)$ admits a minimal graded free resolution $E_\bullet$ as graded $R_X$-module:
		\[
		\cdots\longrightarrow E_2=\bigoplus_j R_X(-a_{2,j})\longrightarrow E_1=\bigoplus_j  R_X(-a_{1,j})\longrightarrow E_0=R_X\oplus \bigoplus_j R_X(-a_{0,j})\longrightarrow R(L)\longrightarrow0.
		\] We say that $L$ satisfies the \emph{Property} $(N_p)$ if $E_0=R_X$ and $a_{i,j}=i+1$ for all $j$ whenever $1\le i\le p.$ In particular, $L$ satisfies $(N_0)$ if and only if $L$ is normally generated. A vector bundle $\E$ on $X$ satisfies the \emph{$(N_p)$ property} if $\OPE(1)$ does it on $\P(\E).$ 
	\end{defi}
	
	\begin{prop}\label{prop:N1-curve}
		Let $C$ be a smooth projective curve of genus $g\ge0$ and let $B$ be a globally generated ample line bundle on $C$ of degree $d>0.$ Let $\E$ be a $B$-Ulrich bundle on $C.$ Then:
		\begin{itemize}
			\item[(1)] $\E$ satisfies $(N_1)$ and $\OPE(1)$ is Koszul if $d>g+2.$
			\item[(2)] $\E$ satisfies $(N_p)$ for $p\ge 2$ if $d>\frac{1}{2} \left((g+p+1)+\sqrt{g^2+2g(3p+1)+(p-1)^2}\right).$
		\end{itemize}
	\end{prop}
	\begin{proof}
		In both situations we have $d>g+1,$ therefore $\E$ is very ample \cite[Theorem 1]{lopez2023geometrical} and semistable with $\mu(\E)=d+g-1$ (Remark \ref{rmk:ulrich-facts}). We also know from Corollary \ref{cor:ulrich-proj-norm-curve} that $\E$ is projectively normal. Letting $\pi\colon \P(\E)\to C$ be the natural projection, through \cite[Exercise III.8.4(a)]{hartshorne2013algebraic} it's immediate to see that $$R^i\pi_\ast\left(\OPE(1)(-1-i)\right)=R^i\pi_\ast\OPE(-i)=0\ \text{for}\ i>0.$$ In other words, $\OPE(1)$ is $(-1)$-regular with respect to $\pi$ in the sense of \cite[Example 1.8.24]{lazarsfeld2017positivity}, or  \cite[§3]{butler1994normal}. Then, if $d>g+2,$ \cite[Theorem 6.1]{butler1994normal} tells that $\OPE(1)$ is Koszul. Since a very ample Koszul line bundle satisfies $(N_1),$ see \cite[Remark 5.2]{butler1994normal} or \cite[Remark 7]{hering2010multigraded}, this proves (1). Finally, (2) is just a rephrasing of \cite[Theorem 1.3]{park2006higher}.
	\end{proof}
	
	\begin{proof}[Proof of Theorem \ref{thm:proj-norm-curve}]
		This is just a recollection of the facts proved in Corollary \ref{cor:ulrich-proj-norm-curve} and Propositions \ref{prop:ulrich-proj-norm}--\ref{prop:ulrich-proj-norm-general}--\ref{prop:N1-curve}
	\end{proof}
	
	Analogously to Green theorem for line bundles on curves \cite[Theorem 1.8.53]{lazarsfeld2017positivity} and to Corollary \ref{cor:ulrich-proj-norm-curve}, the expectation is that an Ulrich bundle on a smooth projective curve of degree $d$ and genus $g$ satisfies $(N_p)$ as soon as $d>g+1+p.$
	
	\section{Projective normality of Ulrich bundles bundles on some surfaces}
	We study the projective normality of $0$-regular vector bundles on smooth regular embedded surfaces. We use a very ample polarization because all the tensor powers remain $0$-regular (Corollary \ref{cor:regular-symmetric}). 
	
	\begin{lemma}\label{lem:proj-norm-0-regular-surface}
		Let $\E$ be a $0$-regular vector bundle on a smooth projective surface $S\subset\P^N.$ Then:
		\begin{itemize}
			\item[(i)]  $\E$ is $k$-normal for all $k\ge 4$;
			\item[(ii)] $\E$ is $3$-normal if $p_g(S)=0$;
			\item[(iii)] If $q(S)=0,$ $\E$ is $2$-normal if and only if $h^2(S,\L^2M_\E)=p_g(S)\binom{h^0(S,\E)}{2}$ if and only if the map $$\L^2H^0(S,\E)^\ast\o H^0(S,K_S)\to H^0(S,\L^2M_\E^\ast\o K_S)$$ is surjective (or, equivalently, an isomorphism). 
		\end{itemize}
	\end{lemma}
	
	Observe that this gives another proof for the fact, already know from Proposition \ref{prop:alzati-russo}, that a $2$-normal ample $0$-regular vector bundle on an embedded smooth projective surface with $p_g=0$ is automatically projectively normal.
	
	\begin{proof}
		We know from (\ref{eq:fact1}) in Lemma \ref{lem:fact} that the syzygy exact sequence $0\to M_\E\to V\o\O_S\to \E\to0,$ where $V=H^0(S,\E),$ yields the long exact sequence $$0\to \L^kM_\E\to\L^kV\o\O_S\to\L^{k-1}V\o\E\to\cdots\to \L^{2}V\o\S^{k-2}\E\to V\o\S^{k-1}\E\to\S^k\E\to0$$ 
		for every $k\ge 2.$ By Remark \ref{rmk:cohomology-serre}, to prove (i) we have to show the surjectivity on global section of the rightmost map.  For $k\ge 4,$ the vanishings  $$H^1(S,\L^2V\o\S^{k-2}\E)\cong \L^2V\o\ H^1(S,\S^{k-2}\E)=0, H^2(S,\L^3V\o\S^{k-3}\E)\cong \L^3V\o H^2(S,\S^{k-3}\E)=0,$$ which follow from the $0$-regularity of the symmetric power of $\E$ (Corollary \ref{cor:regular-symmetric}), together with $$H^i(S,\L^{i+1}V\o\S^{k-i-1}\E)\cong\L^{i+1}V\o H^i(S,\S^{k-i-1}\E)=0\ \text{for}\ i\ge 3,$$ give the claim by \cite[Example B.1.3]{lazarsfeld2017positivity}. 
		
		Now, assuming $p_g(S)=0,$ for $k=3$ we similarly have  $$H^1(S,\L^2V\o\E)\cong \L^2V\o H^1(S,\E)=0, H^2(\L^3V\o\O_S)\cong \L^3H^0(S,\E)\o H^2(S,\O_S)=0$$ and  $H^3(S,\L^3M_\E)=0.$ Item (ii) is obtained in the same way.
		
		Assume $q(S)=0,$  so that $H^1(S,M_\E)=0,$  then consider $k=2$ in the above exact sequence and let $K=\ker(V\o \E\to \S^2\E).$ The cohomology of the corresponding exact sequence immediately shows that $H^2(S,K)\cong H^1(S,\S^2\E)=0$ and that $\E$ is $2$-normal if and only if $H^1(S,K)=0.$ On the other hand, as $K=\mathrm{Im}(\L^2M_\E\to\L^2V\o\O_S),$ we get the exact sequence of vector spaces $$0\to H^1(S,K)\to H^2(S,\L^2M_\E)\xrightarrow{f} \L^2V\o H^2(S,\O_S)\to H^2(S,K)=0.$$ Thus, $f$ being surjective, $H^1(S,K)=0$ if and only if $f$ is injective if and only if $f$ is an isomorphism if and only if $h^2(S,\L^2M_\E)=p_g(S)\cdot \dim \L^2V.$ Taking the dual map of $f$ and using  Serre duality, we see that $f$ is injective if and only if $$\L^2V^\ast\o H^0(S,K_S)\to H^0(S,\L^2M_\E^\ast
		\o K_S)$$ is surjective.  This proves (iii).
	\end{proof}
	
	In virtue of this, we will mostly focus on regular smooth surfaces $S$ possibly with $p_g(S)=0.$
	
	Let's see what happens when we consider $0$-regular locally free sheaves with aCM projectivized bundle.
	
	\begin{prop}\label{prop:0-regular-aCM-surface}
		Let $S\subset\P^N$ be smooth projective surface and let $\E$ be a very ample $0$-regular bundle on $S$ of rank $r\ge2.$ Then:
		\begin{itemize}
			\item[(1)] If $\P(\E)$ is aCM, then $q(S)=p_g(S)=0.$
			\item[(2)] If $q(S)=p_g(S)=0$ and $\E$ is $2$-normal, then $\P(\E)$ is aCM, $\I_{\P(\E)/\P(H^0(S,\E))}$ is $3$-regular and $I_{\P(\E)/\P(H^0(S,\E))}$ is generated in degree less than or equal to $3.$
		\end{itemize}
	\end{prop}
	
	\begin{proof}
		Since $\S^k\E$ is $0$-regular (Corollary \ref{cor:regular-symmetric}), by \cite[Exercise III.8.4]{hartshorne2013algebraic} we have
		\begin{equation}\label{eq:cohomology}
			H^i(\P(\E),\OPE(\l))=\begin{cases}
				H^i(S,\O_S)& \mbox{for}\ i\ge0,\l=0\\
				0&\mbox{for}\ i\ge1,\l\ge-r+1,\l\ne0.
			\end{cases}
		\end{equation} Therefore, if $\P(\E)$ is aCM, then $H^i(S,\O_S)=H^i(\P(\E),\OPE)=0$ for $1\le i\le 2$ (Remark \ref{rmk:acm-variety}). This gives (1).
		
		Now assume $q(S)=p_g(S)=0$ and that $\E$ is $2$-normal. Then Proposition \ref{prop:alzati-russo} tells that $\E$ is projectively normal with $\I_{\P(\E)/\P(H^0(S,\E))}$ which is $3$-regular and $I_{\P(\E)/\P(H^0(S,\E))}$ that is generated in degree less than or equal to $3.$ We only need to prove that $\P(\E)$ is aCM. To do this, by Lemma \ref{lem:proj-norm-hyperplane-section} it suffices to prove that a smooth sectional curve of $\P(\E)\subset\P(H^0(S,\E))$ is projectively normal. Take $r$ smooth hyperplane sections $H_1,\dots, H_r\in |\OPE(1)|$ such that $Y_j=H_1\cap\cdots\cap H_j\subset\P(\E)$ is smooth and irreducible. We claim that $H^i(Y_j,\O_{Y_j}(h))=0$ for all $i\ge1$ as soon as $h\ge-r+1+j.$ \\
		To see this, we proceed by induction on $1\le j\le r.$ Thanks to (\ref{eq:cohomology}), we can immediately see from the exact sequence $0\to\OPE(h-1)\to\OPE(h)\to\O_{Y_1}(h)\to0$ that $H^i(Y_1,\O_{Y_1}(h))=0$ for all $i\ge1$ as long as $h-1\ge-r+1$ as desired. For $j>1,$ we get the claim by applying the inductive hypothesis to $0\to\O_{Y_{j-1}}(h-1)\to\O_{Y_{j-1}}(h)\to\O_{Y_{j}}(h)\to0.$ \\
		Now, as $H^1(\P(\E),\OPE)=H^1(Y_j,\O_{Y_j})=0$ for all $1\le j\le r-1,$ we know that all $Y_1,\dots,Y_r$ are linearly normal. Moreover, since $H^1(\P(\E),\OPE(k-1))=H^1(Y_j,\O_{Y_j}(k-1))=0$ for all $k\ge 2$ and $1\le j<r,$ Lemma \ref{lem:proj-norm-hyperplane-section} tells that $Y_1,\dots,Y_r$ are projectively normal. Since $Y_r$ is a sectional curve, this proves (2).
	\end{proof}
	
	For $0$-regular vector bundles with ample determinant on embedded surfaces with $q=p_g=0$ we can also give a geometric characterization for the non-$2$-normality.
	
	\begin{prop}\label{prop:proj-norm-0-regular-q=p=0}
		Let $S\subset\P^N$ be smooth projective surface with $q(S)=p_g(S)=0$ and let $\E$ be a $0$-regular vector bundle of rank $r\ge 2$ on $S$ with ample determinant line bundle bundle $E=\det(\E).$ Assume $h=h^0(S,\E)\ge r+3$ and let $\l=\binom{h-r}{2}-1.$ The following are equivalent:
		\begin{itemize}
			\item[(1)] $\E$ is not $2$-normal.
			\item[(2)] There exist a closed subscheme $Z\subset S$ and a non-zero divisor $D\subset S$ such that:
			\begin{itemize}
				\item[(a)] $Z$ is smooth of dimension $0.$
				\item[(b)] $Z$ is the degeneracy locus of $\l$ general sections $s_1,\dots,s_\l\in H^0(S,\L^2M_\E^\ast).$
				\item[(c)] $[Z]=\frac{1}{2}(h-r-2)\left((h-r+1)c_1(\E)^2-2c_2(\E)\right).$
				\item[(d)] $D\in |K_S+(h-r-1)E|.$
				\item[(e)] $Z\subset D.$
			\end{itemize}
			\item[(3)] There exist a closed subscheme $Z\subset S$ and a curve $C\subset S$ such that: 
			\begin{itemize}
				\item[(f)] $Z$ is the degeneracy locus of $\l$ general sections $\s_1,\dots,\s_\l\in  H^0(S,\L^2M_\E^\ast).$
				\item[(g)] $C$ is the degeneracy locus of the $(\l+1)$ general sections $\s_1,\dots,\s_\l,\s_{\l+1}\in  H^0(S,\L^2M_\E^\ast).$
				\item[(h)] $C\in |(h-r-1)E|$ is smooth and irreducible.
				\item[(i)] $Z\subset C$ is a special (effective) divisor.
			\end{itemize}
		\end{itemize}
	\end{prop}
	
	\begin{proof}
		First of all, from the syzygy exact sequence $0\to M_\E\to H^0(S,\E)\o\O_S\to\E\to 0$ and its dual we can see that:
		\begin{enumerate}[(A)]
			\item $H^0(S,\L^2M_\E)\subset H^0(S,M_\E\o M_\E)\subset H^0(S,\E)\o H^0(S,M_\E)=0$; 
			\item $c_1(M_\E)=-c_1(\E)$ and $c_2(M_\E)=c_1(\E)^2-c_2(\E)$;
			\item the map $\L^2H^0(S,\E)^\ast\o\O_S\to\L^2M_\E^\ast$ is surjective (thanks to (\ref{eq:fact3}) for $k=2$). 
		\end{enumerate}
		Now assume (1).	Since (C) tells in particular that $\L^2M_\E^\ast$ is generated by global sections, we can take a general subspace $V\subset H^0(S,\L^2M_\E^\ast)$ of dimension $$\l=\rk(\L^2M_\E^\ast)-1=\binom{h-r}{2}-1$$ and we let $Z=D_{\l-1}(v)\subset S$ be the degeneracy locus of the evaluation map $v\colon V\o\O_S\to \L^2M_\E^\ast.$ As $\L^2M_\E^\ast$ is generated by global sections and $\dim S=2,$ it is well known, for instance from \cite[Statement (folklore), §4.1]{banica1989smooth}, that $Z$ is either empty or is smooth of (the expected) codimension $2$ and that there is a short exact sequence 
		\[
		\begin{tikzcd}
			0\rar& V\o\O_S\rar&\L^2M_\E^\ast\rar&\I_{Z/S}(M)\rar&0,
		\end{tikzcd}
		\] with $M=\det(\L^2M_\E^\ast)=(h-r-1)E$ (Lemma \ref{app:exterior-chern}(i)) and $\I_{Z/S}=\O_S$ in case $Z=\emptyset.$ However $Z$ cannot be empty, for otherwise, by dualizing the above sequence and using Kodaira vanishing on the line bundle $M,$ which is ample by the hypothesis, we would have $H^0(S,\L^2M_\E)\cong V^\ast$ which contradicts (A). This gives (a)-(b). As $Z\neq\emptyset,$ by construction and (B) (and by Lemma \ref{app:exterior-chern}) we get  $$[Z]=c_2(\L^2M_\E^\ast)=\frac{1}{2}(h-r+1)(h-r-2)c_1(\E)^2-(h-r-2)c_2(\E),$$ that is (c). Now, twist the above sequence through by $K_S$ and take the cohomology. We immediately see that $H^0(S,\L^2M_\E^\ast\o K_S)\cong H^0(S,\I_{Z/S}(M+K_S)).$ By Lemma  \ref{lem:proj-norm-0-regular-surface}, we finally get (d)-(e), proving (2).
		
		Conversely, if $Z\subset S$ satisfies  (a) and (b), then (c) holds by construction. Furthermore, by \cite[Teorema 2.14 \& Esercizio, p. 23]{ottaviani1995varieta}, the ideal sheaf $\I_{Z/S}$ fits into the same exact sequence as above. If there is $D\subset S$ satisfying (d)-(e), we have $H^0(S,\I_{Z/S}(M+K_S))\ne 0.$ Repeating the above argument backwards, we obtain the equivalence between (1) and (2).
		
		Now we observe that (a) and (b), which imply (c) as seen above, is equivalent to (f) and (g), which in turn yield (h) by construction. One direction is obvious, hence we assume (b). Let $W\subset  H^0(S,\L^2M_\E^\ast)$ be the subspace generated by the sections in (b) and choose a subspace $W'\subset  H^0(S,\L^2M_\E^\ast)$ of dimension $\l+1$ that contains $W.$ The degeneracy locus $D_{\l+1}(w)$ of the evaluation map $w\colon W'\otimes \O_S\to \L^2M_\E^\ast$ contains $Z,$ is supported on a smooth member $C\in |M|,$ which is then irreducible by \cite[Corollary III.7.9]{hartshorne2013algebraic}, and $\mr{coker}(w)=L$ is a line bundle on $C$ (see \cite[Statement (folklore), §4.1]{banica1989smooth} and \cite[§4.1]{benedetti2024conjecture}). This gives the claim. In addition to this, from the Snake lemma applied to the diagram 
		\[\begin{tikzcd}
			0\rar& W\o\O_S\rar&\L^2M_\E^\ast\rar&\I_{Z/S}(M)\rar&0, \\
			\\
			0 & W'\o\O_S & \L^2M_\E^\ast & L & 0,
			\arrow[from=1-1, to=1-2]
			\arrow[from=1-2, to=1-3]
			\arrow[hook, from=1-2, to=3-2]
			\arrow[from=1-3, to=1-4]
			\arrow[from=1-3, to=3-3, "\cong"]
			\arrow[from=1-4, to=1-5]
			\arrow[from=1-4, to=3-4]
			\arrow[from=3-1, to=3-2]
			\arrow[from=3-2, to=3-3]
			\arrow[from=3-3, to=3-4]
			\arrow[from=3-4, to=3-5]
		\end{tikzcd}\]
		we get the exact sequence $$
		\begin{tikzcd}
			0\rar&\O_S\rar&\I_{Z/S}(M)\rar& L\rar& 0
		\end{tikzcd}
		$$ which yields $L(-M)=i_\ast\O_C(-Z),$ where $i\colon C\hookrightarrow S$ is the inclusion.
		
		To prove the equivalence (2)-(3), it's enough to see that $$H^1(C,Z)\ne 0\ \text{if and only if}\ H^0(S,\I_{Z/S}(M+K_S))\ne 0.$$ To do this, twist the above sequence through by $K_S$ and take the cohomology. It's immediate to see that $H^0(S,\I_{Z/S}(M+K_S))\cong H^0(C,\O_C(K_S+M-Z)).$ On the other hand, by adjunction and Serre duality, we have $h^0(C,\O_C(K_S+M-Z))=h^1(C,Z).$ Then the assertion follows. 
	\end{proof}
	
	In turn we get a characterization for ample $0$-regular bundle on embedded surfaces with $q=p_g=0$ for non-being projectively normal and aCM.
	
	\begin{proof}[Proof of Theorem \ref{thm:proj-norm-0-regular-q=p=0}]
		Remark \ref{rmk:acm-variety} immediately tells that (2) implies (1). For the converse, we just need to observe that if $\E$ was projectively normal, then it would be automatically very ample and $2$-normal, but then Proposition \ref{prop:0-regular-aCM-surface} would imply that $\P(\E)$ is aCM, contradicting the assumption. 
		On the other hand, since $\E$ is ample, by Remark \ref{rmk:proj-norm} and Lemma \ref{lem:proj-norm-0-regular-surface} we know that $\E$ is not projectively normal if and only if $\E$ is not $2$-normal. As $E$ is ample as well by \cite[Corollary 6.1.16]{lazarsfeld2017positivity2}, the equivalence between (2), (3) and (4) follows from Proposition \ref{prop:proj-norm-0-regular-q=p=0}.
	\end{proof}
	
	Then we can provide a numerical criterion for the non projective normality of an ample $0$-regular bundle on a surface with $q=p_g=0.$ 
	
	\begin{cor}\label{cor:non-proj-norm-surf-q=p=0}
		Let $S\subset\P^N$ be a smooth projective surface with $q(S)=p_g(S)=0.$ Let $\E$ be an ample $0$-regular vector bundle on $S\subset\P^N$ of rank $r\ge2$ with $h=h^0(S,\E)\ge r+3.$ Then $\E$ is not aCM, or equivalently not projectively normal, if 
		\[
		(h-r-1)c_1(\E)\cdot K_S+2(h-r-2)c_2(\E)+h(h-1)>(h-r-3)c_1(\E)^2+r(2h-r-1).
		\] 
	\end{cor}
	
	We first state a simple remark involving Grothendieck-Verdier duality.
	
	\begin{rmk}\label{rmk:grothendieck-verdier}
		Let $f\colon X\to Y$ be an affine morphism of smooth projective varieties of relative dimension $-k=\dim X-\dim Y\le0,$ and let $\F,\g$ be locally free sheaves on $X,Y$ respectively. Then $$\EExt^{k}_{\O_Y}(f_\ast\F,\g)\cong \g(-K_Y)\o f_\ast(\F^\ast(K_X)).$$ In particular, if $f$ is the inclusion of a subvariety of codimension $k\ge1$ with normal bundle $\mc{N}_{X/Y},$ then  $$\EExt^k_{\O_Y}(f_\ast\F,\g)\cong \g\o f_\ast(\F^\ast\o\L^k\mc{N}_{X/Y}).$$
		
		To see this, we are going to use Grothendieck-Verdier duality \cite[Theorem 3.34]{huybrechts2006fourier}. By \cite[Exercise III.8.2]{hartshorne2013algebraic} the direct image $\mb{R}f_\ast\E=f_\ast\E$ does not need to be derived for any coherent sheaf $\E$ on $X,$ as well as  $\mb{L}f^\ast(\g)=f^\ast\g$ because $\g$ is locally free. The relative dualizing bundle of $f$ is $\omega_f=K_X(-f^\ast K_S),$ hence Grothendieck-Verdier duality yields the isomorphisms
		\begin{equation}\label{eq:grothendieck-verdier-duality}
			\begin{aligned}
				\RHom_{\O_Y}\left(f_\ast\F,\g\right)&\cong \RHom_{\O_Y}\left(\mb{R}f_\ast\F,\g\right)\\
				&\cong\mb{R}f_\ast\RHom_{\O_X}\left(\F,\mb{L}i^\ast(\g)\o\omega_f\left[-k\right]\right)\\
				&\cong\mb{R}f_\ast\RHom_{\O_X}\left(\F,f^\ast\g\o\omega_f\left[-k\right]\right)\\
				&\cong\mb{R}f_\ast\Hom_{\O_X}\left(\F,f^\ast\g\o\omega_f[-k]\right)\\
				&\cong\mb{R}f_\ast(\F^\ast\o f^\ast\g\o\omega_f[-k])\\
				&=\mb{R}f_\ast(\F^\ast(K_X)\o f^\ast(\g(-K_Y))[-k]).
			\end{aligned} 
		\end{equation}
		Here we performed the identifications $$\RHom_{\O_X}\left(\F,f^\ast\g\o\omega_f[-k]\right)\cong \Hom_{\O_X}\left(\F,f^\ast\g\o\omega_f[-k]\right)\cong\F^\ast\o f^\ast\g\o\omega_f[-k]$$ as complexes concentrated in degree $k$ which comes from the locally freeness on $X$ of $\F$ (see \cite[§3.3, p. 84]{huybrechts2006fourier} and \cite[Exercise II.5.1(b)]{hartshorne2013algebraic}). 
		Taking the $k$-th cohomology sheaf $\mc{H}^k\colon\mr{D}^{\mr{b}}(\mb{Coh}(Y))\to\mb{Coh}(Y)$ on both sides of (\ref{eq:grothendieck-verdier-duality}) and adopting the usual sign conventions for translations and cohomology \cite[(1.3.4)]{conrad2000grothendieck}, by definition \cite[(21.21.2)]{gortz2023algebraic} and by \cite[Remark 21.26(3)]{gortz2023algebraic} we have the claimed isomorphism of coherent sheaves
		\begin{align*}
			\EExt^k_{\O_Y}(f_\ast\F,\g)&\cong \mc{H}^k(\RHom_{\O_Y}\left(f_\ast\F,\g\right))\\
			&\cong\mc{H}^k(\mb{R}f_\ast(\F^\ast(K_X)\o f^\ast(\g(-K_Y))[-k]))\\
			& \cong \mb{R}^0(f_\ast(\F^\ast(K_X)\o f^\ast(\g(-K_Y))))\\
			&\cong f_\ast(\F^\ast(K_X)\o f^\ast(\g(-K_Y)))\\
			&\cong \g(-K_Y)\o f_\ast(\F^\ast(K_X)).
		\end{align*}
		The last part of the claim descends from adjunction formula \cite[Proposition II.8.20]{hartshorne2013algebraic}.
	\end{rmk}

	\begin{proof}[Proof of Corollary \ref{cor:non-proj-norm-surf-q=p=0}]
		Just like in Proposition \ref{prop:proj-norm-0-regular-q=p=0} and its proof, consider the degeneracy loci $Z\subset C\subset S$ of $\l$ and $\l+1$ general sections of the globally generated bundle $\L^2M_\E^\ast,$ with $\l=\binom{h-r}{2}-1.$ Both of them are smooth, $Z$ is $0$-dimensional with $[Z]$ given in Proposition \ref{prop:proj-norm-0-regular-q=p=0}(c), and $C\in |M|$ is irreducible, where $M=\det(\L^2M_\E^\ast)=(h-r-1)\det(\E).$  
		As seen in the proof of Proposition \ref{prop:proj-norm-0-regular-q=p=0}, we have the exact sequence 
		\[
		\begin{tikzcd}
			0\rar&\O_S^{\oplus(\l+1)}\rar&\L^2M_\E^\ast\rar&i_\ast\O_C(M-Z)\rar& 0
		\end{tikzcd}
		\] where $i\colon C\hookrightarrow S$ is the inclusion. Taking the dual with respect to $S,$ we get the short exact sequence
		\[
		\begin{tikzcd}
			0\rar&\L^2M_\E\rar&\O_S^{\oplus(\l+1)}\rar&\EExt^1_{\O_S}(i_\ast\O_C(M-Z),\O_S)\rar&0.
		\end{tikzcd}
		\] But $\EExt^1_{\O_S}(i_\ast\O_C(M-Z),\O_S)\cong i_\ast\O_C(Z)$ by Remark \ref{rmk:grothendieck-verdier}.	Thus, using that $\L^2M_\E$ has no non-zero global sections, we can see that $h^0(C,Z)\ge \l+1.$ Computing the genus of $C\subset S$ via adjunction formula, applying Riemann-Roch theorem and using this inequality, we see that
		\[
		h^1(C,Z)\ge \frac{1}{2}\left((h-r-1)c_1(\E)\cdot K_S+2(h-r-2)c_2(\E)+h(h-1)-(h-r-3)c_1(\E)^2-r(2h-r-1)\right).
		\] By Proposition \ref{prop:proj-norm-0-regular-q=p=0}, $\E$ is not projectively normal as soon as $h^1(C,Z)>0.$ Then the claim follows.
	\end{proof}
	
	For an Ulrich bundle $\E$ on a smooth embedded surface $S\subset\P^N$ we can calculate the dimensions of the global sections of the second tensor power $\E\o\E$ and of the symmetric powers $\S^2\E,\S^3\E.$ As a necessary condition for the $k$-normality of $\E$ is $\dim\S^kH^0(S,\E)\ge h^0(S,\S^k\E),$ this gives a way to check if $\E$ could potentially be $k$-normal. This method will be used in the next examples and also in the next section for Ulrich bundles on hypersurfaces.
	
	First we recall that the second Chern class of an Ulrich bundle $\E$ of rank $r$ on a smooth surface $S\subset\P^N$ of degree $d$ is given by Casnati formula \cite[Proposition 2.1]{casnati2017special}
	\begin{equation}\label{eq:ulrich-surface-c2}
		c_2(\E)=\frac{1}{2}\left(c_1(\E)^2-c_1(\E)\cdot K_S\right)+r\chi(\O_S)-rd.
	\end{equation}
	
	\begin{lemma}\label{lem:ulrich-global-sections-tensor-surface}
		Let $\E$ be a rank $r$ Ulrich bundle on a smooth projective surface $S\subset\P^N$ of degree $d.$ Then:\begin{itemize}
			\item[(i)] $h^0(S,\E\o\E)=\chi(S,\E\o\E)=c_1(\E)^2+r^2(2d-\chi(S,\O_S)).$
			\item[(ii)] $h^0(S,\S^2\E)=\chi(S,\S^2\E)=r(r+2)d+\frac{1}{2}(c_1(\E)^2+c_1(\E)\cdot K_S)-\frac{r(r+3)}{2}\chi(S,\O_S).$
			\item[(iii)] $h^0(S,\S^3\E)=\chi(S,\S^3\E)=\frac{1}{6}(r+2)(3c_1(\E)^2+3c_1(\E)\cdot K_S+3rd(r+3)-2r\chi(S,\O_S)(r+4)).$
		\end{itemize}
		In particular, if $c_1(\E)=\frac{r}{2}(K_S+3H),$ where $H$ is the class of a hyperplane section, we have:
		\begin{itemize}
			\item[(a)] $h^0(\E\o\E)=\chi(S,\E\o\E)=\frac{r^2}{4} \left[17d+6 K_S\cdot H+K_S^2-4\chi(S,\O_S)\right].$
			\item[(b)] $h^0(S,\S^2\E)=\chi(S,\S^2\E)=\frac{r}{8}\left[(17r+16)d+(2+r)K_S^2+6(r+1)K_S\cdot H-4(r+3)\chi(S,\O_S)\right].$
			\item[(c)] $h^0(S,\S^3\E)=\chi(S,\S^3\E)\\
			=\frac{r}{24}(r+2)\left[3d(13r+12)+3(r+2)K_S^2+18(r+1)H\cdot K_S-8(r+4)\chi(S,\O_S)\right].$
		\end{itemize}
	\end{lemma}
	\begin{proof}
		Both $\E\o\E$ and the symmetric powers $\S^2\E,\S^3\E$ are $0$-regular (Corollary \ref{cor:regular-symmetric}), therefore we only need to compute the Euler characteristics $\chi(S,\E\o\E)$ and $\chi(S,\S^2\E), \chi(S,\S^3\E).$ This is easily done via Hirzebruch-Riemann-Roch theorem \cite[Theorem A.4.1]{hartshorne2013algebraic} which says that 
		\[
		\chi(S,\F)=s\chi(S,\O_S)+\frac{1}{2}\left(c_1(\F)^2-c_1(\F)\cdot K_S\right)-c_2(\F)
		\] for any rank $s$ vector bundle $\F$ on $S.$ Using (\ref{eq:ulrich-surface-c2}), the conclusion follows from formulae in Lemmas \ref{app:tensor-chern}-\ref{app:S^k-S^3} and Corollary \ref{app:symmetric-chern}.
	\end{proof}
	
	\begin{ex}\label{ex:ulrich-k3}
		Let $S\subset\P^3$ be any smooth K3 surface of degree $4$ and let $\E$ be an Ulrich bundle of rank $r=2$ or $r=4$ with $\det(\E)=\O_S(3r/2)$ (which always exists by \cite[Theorem 1]{faenzi2019ulrich} and \cite[Proposition 3.3.1]{costa2021ulrich}. Then $\E$ cannot be projectively normal because, by Lemma \ref{lem:ulrich-global-sections-tensor-surface}(b), the inequality $$\dim \S^2 H^0(S,\E)=2r(4r+1)\ge h^0(S,\S^2\E),$$ which is necessary for the $2$-normality, is never satisfied.
	\end{ex}
	
	\begin{ex}\label{ex:ulrich-complete-intersection}
		Let $S\subset\P^4$ be a smooth very general complete intersection of type $(2,a)$ for $a\ge 3.$ Then no rank $r$-Ulrich bundle on $S,$ which always exists by \cite{herzog1991linear}, can be projectively normal if $a\ge 15.$
		
		To see this, let $\E$ be an Ulrich bundle of rank $r$ on $S.$ By Noether-Lefschetz theorem we can assume $\Pic(S)\cong \Z\cdot\O_S(1)$ (for instance, one can apply \cite[Theorem 1]{ravindra2009noether}, or \cite[Proposition 3.2]{bruzzo2021existence}, to $\O_Q(a)$ where $Q\subset\P^4$ is a smooth quadric). Letting $d:=\deg S=2a$ be the degree of $S$ and $H\subset S$ be a hyperplane section of $S\subset\P^N,$ we have $K_S=(a-3)H=\frac{1}{2}(d-6)H$ and $\chi(S,\O_S)=\frac{d}{24}(d^2-9d+26).$ Moreover, we know from Remark \ref{rmk:ulrich-facts} that $c_1(\E)=\frac{r}{2}(K_S+3H)=\frac{r}{4}dH.$ Substituting in Lemma \ref{lem:ulrich-global-sections-tensor-surface}(b) we get $$h^0(S,\S^2\E)=\frac{rd}{96}\left[rd^2+18(r+1)d+44r+36\right].$$ This means that if $\E$ is $2$-normal, then $h^0(S,\S^2\E)-rd(rd+1)/2\le0,$ which happens if and only if $$rd^2-(30r-18)d+44r-12\le0.$$ One can check that this forces the following conditions on $r$ and $d=2a$:
		\begin{itemize}
			\item $6\le d\le 18$ if $r\ge 2$; 
			\item $d=20,22$ if $r\ge 3$;
			\item $d=24$ if $r\ge 5$;
			\item $d=26$ if $r\ge 8$;
			\item $d=28$ if $r\ge 41$;
		\end{itemize} In any case we have $2a=d\le 28.$ Then the claim follows.
	\end{ex}

	\section{Projective normality of Ulrich bundles on low dimensional hypersurfaces}
	Ulrich bundles on hypersurfaces always exist \cite{herzog1991linear}. Moreover, Ulrich bundles on the general hypersurface $X\subset\P^{n+1}$ of degree $d\ge 2n$ is very ample as $X$ does not contain lines \cite[Theorem1 ]{lopez2023geometrical}. Therefore the expectation is them to be quite positive, in particular projectively normal. However we will see that this is not the case for hypersurfaces of dimension $2$ and $3.$	In this section, all hypersurfaces are smooth.\\
	
	We start by recalling that every Ulrich bundle on a hypersurface $X\subset\P^{n+1}$ admits a locally free resolution on $X.$
	
	\begin{lemma}\label{lem:ulrich-tensor-res}
		Let $X\subset\P^{n+1}$ be a smooth hypersurface of dimension $n\ge1$ and degree $d\ge2.$ An Ulrich bundle $\E$ of rank $r$ on $X$ admits the resolution
		\begin{equation}\label{eq:ulrich-resolution}
			\begin{tikzcd}
				0\rar&\E(-d)\rar&\O_X(-1)^{\oplus rd}\rar&\O_X^{\oplus rd}\cong H^0(X,\E)\o\O_X\rar&\E\rar&0.
			\end{tikzcd}
		\end{equation}
		In particular, the tensor power $\E\o\E$ is resolved by
		\begin{equation}\label{eq:ulrich-tensor-res}
			\begin{tikzcd}
				0\rar&\E\o\E(-d)\rar& \E(-1)^{\oplus rd}\rar& \E^{\oplus rd}\cong H^0(X,\E)\o\E\rar& \E\o\E\rar&0.
			\end{tikzcd}
		\end{equation}
	\end{lemma}
	\begin{proof}
		This is essentially \cite[§2]{tripathi2016splitting} and \cite[§2]{tripathi2017rank} applied to the resolution \[
		\begin{tikzcd}
			0\rar&\O_{\P^{n+1}}(-1)^{\oplus rd}\rar&\O_{\P^{n+1}}^{\oplus rd}\cong H^0(\P^{n+1},i_\ast\E)\o\O_{\P^{n+1}}\cong H^0(X,\E)\o\O_{\P^{n+1}}\rar&i_\ast\E\rar&0
		\end{tikzcd}
		\] of the extension $i_\ast\E$ on $\P^{n+1},$ with $i\colon X\hookrightarrow\P^{n+1}$ being the inclusion, given by the definition of Ulrich bundle (see \cite[Proposition 2.1 \& Theorem 2.3]{beauville2018introduction}). 
	\end{proof}
	
	\begin{rmk}
		Despite the embedding of every smooth quadric $Q\subset\P^{n+1}$ is Koszul (Example \ref{ex:koszul}), which implies that every $0$-regular vector bundle on $Q$ is $k$-normal for all $k\ge1,$ an Ulrich bundle on $Q$ cannot be projectively normal if $n\ge3$ because it is not ample by \cite[Theorem 1]{lopez2023geometrical} and \cite[Corollary 1.6]{ottaviani1988spinor}.
	\end{rmk}
	
	\begin{lemma}\label{lem:ulrich-tensor-2d-hyper}
		Let $S\subset\P^3$ be a smooth hypersurface of degree $d\ge2$ and let $\E$ be an Ulrich bundle of rank $r$ on $X$ such that $\det(\E)=\O_X(\frac{r}{2}(d-1)).$  Then:
		\begin{itemize}
			\item[(1)] $h^0(S,\E\o\E)=\chi(S,\E\o\E)=\frac{r^2d}{12}(d+1)(d+5).$
			\item[(2)] $h^0(S,\S^2\E)=\chi(S,\S^2\E)=\frac{rd}{24}\left(d+1\right)\left((d+5)r+6\right).$
			\item[(3)] $h^0(S,\S^3\E)=\chi(S,\S^3\E)=\frac{1}{72}rd(d+1)(r+2)(r+4+d(5r+2)).$
		\end{itemize} 
		In particular, if $\E$ is $2$-normal, then either $d=2,$ or $d=3$ and $r\ge3,$ or $d=4$ and $r\ge6.$
	\end{lemma}
	
	Note that, by Remark \ref{rmk:ulrich-facts}, the set of $2$-dimensional smooth hypersurfaces supporting such Ulrich bundles contains the subset of hypersurfaces $S\subset\P^3$ with $\Pic(S)\cong\Z\cdot\O_S(1),$ which are very general in $|\O_{\P^3}(d)|$ by Noether-Lefschetz theorem \cite{lefschetz1924analysis, griffiths1985noether}.
	
	\begin{proof}
		By what said in Remark \ref{rmk:ulrich-facts}, the assumption on the Picard group forces $r\ge2.$ Denoting the class of a hyperplane section of $S\subset\P^3$ by $H,$ then $c_1(\E)=\frac{r}{2}(d-1)H$ (Remark \ref{rmk:ulrich-facts}).  Since
		$$\chi(S,\O_S)=1+\binom{d-1}{3}=\frac{d(d^2-6d+11)}{6},$$  formulae (1)-(2)-(3) are obtained by substituting these in Lemma \ref{lem:ulrich-global-sections-tensor-surface}(a-b-c). For the last part, observe that the $2$-normality of $\E$ implies the inequality
		$$\dim \S^2H^0(S,\E)=\frac{rd(rd+1)}{2}<h^0(X,\S^2\E),$$ which is satisfied in the claimed ranges.
	\end{proof}
	
	On the contrary, in the situation of the above Lemma, the difference 
	$$\dim\S^3H^0(S,\E)-h^0(S,\S^3\E)=\frac{rd}{72} (d-1) (r-2) (d (7 r+2)+r+8)$$ is negative if and only if $r=1$ (which happens very rarely), it's zero for rank $2$ (that is for Pfaffian surfaces), and it is always positive for $r\ge3.$ Therefore the $3$-normality of $\E$ usually does not imposes any restriction.
	
	We now move to hypersurfaces in $\P^4.$ We begin with the calculations of Chern classes and Euler characteristics of tensor powers and symmetric powers. Note that it is no longer granted that the tensor operations of Ulrich bundles are again $0$-regular.
	
	\begin{lemma}\label{lem:ulrich-tensor-3d-hyper}
		Let $X\subset\P^4$ be a smooth hypersurface of degree $d\ge1$ and let $\E$ be a rank $r$ Ulrich bundle on $X.$ Then:
		\begin{itemize}
			\item[(i)] $c_3(\E)=\frac{rd}{48}(d-1)^2(r-2)(rd-r+2).$
			\item[(ii)] $c_3(\E\o\E)=\frac{r^2d}{12} (d-1)^2 \left(r^2-2\right) \left(2r^2(d-1)+3-d\right).$
			\item[(iii)] $c_3(\S^2\E)=\frac{rd}{48} (d-1)^2 (r+2) \left(r^2+r-4\right) \left(r^2(d-1)+2\right).$
			\item[(iv)] $\chi(X,\E\otimes \E)=h^0(X,\E\o\E)-h^1(X,\E\o\E)=\frac{r^2d}{8}(d+1)(d+3).$
			\item[(v)] $\chi(X,\S^2\E)=h^0(X,\S^2\E)-h^1(X,\S^2\E)=\frac{rd}{48}(d+1)(d+3)(3r+4-d).$
		\end{itemize}
	\end{lemma}
	
	\begin{proof}
		For (i),  see \cite[Proposition 3.7]{benedetti2023projective} or just apply Hirzebruch-Riemann-Roch theorem to $\chi(X,\E)=rd.$  Lefschetz theorem tells that the restriction map $\Pic(\P^4)\to\Pic(X)$ induces an isomorphism \cite[Example 3.1.25]{lazarsfeld2017positivity}. Hence, letting $H$ be the class of a hyperplane section of $X\subset\P^4,$  we have $c_1(\E)=\frac{r}{2}(d-1)H $ (Remark \ref{rmk:ulrich-facts}).  Using \cite[Lemma 3.2(ii)]{lopez2023varieties} and (i) on Lemma \ref{app:tensor-chern}(iii) and Corollary \ref{app:symmetric-chern}(iii) one obtains (ii) and (iii) respectively. Now, observe that 
		\begin{equation}\label{eq:ulrich-tensor-vanishing-3d-hyper}
			H^2(X,\E\o\E(p))=H^3(X,\E\o\E(p))=0\hspace{1cm} \mbox{for}\ p\ge-2
		\end{equation} by \cite[Proposition B.1.2(i)]{lazarsfeld2017positivity} applied to the 	resolution (\ref{eq:ulrich-tensor-res}) twisted by $\O_X(p).$ Then (iv) and (v) follow from Hirzebruch-Riemann-Roch theorem for $\E\o\E$ and $\S^2\E$ respectively.
	\end{proof}
	
	\begin{proof}[Proof of Theorem \ref{thm:projective-normality-hypersurface}]
		Item (a) is an immediate consequence of Lemma \ref{lem:ulrich-tensor-2d-hyper}. For (b), note that we must have $r\ge2$ by Lefschetz theorem \cite[Example 3.1.25]{lazarsfeld2017positivity} and Remark \ref{rmk:ulrich-facts}. Thanks to Lemma \ref{lem:ulrich-tensor-3d-hyper}(iv), we immediately see that the inequality $$h^0(X,\E\o\E)\ge\chi(X,\E\o\E)>\dim H^0(X,\E)^{\o2}=r^2d^2$$ holds for every pair $(r,d).$ Therefore $\mu_\E$ is never surjective. Finally, by Lemma \ref{lem:ulrich-tensor-3d-hyper}(v), we can check that $$h^0(X,\S^2\E)\ge\chi(X,\S^2\E)>\dim \S^2 H^0(X,\E)=\frac{rd(rd+1)}{2}$$ holds for if $r>\frac{d+4}{3}.$ In particular $\E$ cannot be projectively normal for all such ranks.
	\end{proof}
	
	The bound on the rank, possibly not optimal, is necessary. As we are going to see, Theorem \ref{thm:projective-normality-hypersurface}(b) is sharp for the $2$-normality of Ulrich bundles on $3$-dimensional hypersurfaces of degree $5.$
	
	\begin{rmk}\label{rmk:hyper}
		Let $X\subset\P^{n+1}$ be a hypersurface of dimension $n=3$ or $4$ and degree $d\ge3.$ Suppose $X$ supports an Ulrich bundle $\E$ of rank $r=2$ or $r=3.$ Then $\E$ is $2$-normal. If $n=4$ and $r=3,$ then $\E$ is also $3$-normal.\footnote{Except for $r=2$ with $n=3$ and $d\le 5,$ these hypersurfaces live in a closed subset of $|\O_{\P^{n+1}}(d)|$ (see \cite[Proposition 8.9]{beauville2000determinantal} and \cite[Corollary 1 \& Theorem 2]{lopez2024ulrich}). Moreover, a pfaffian hypersurface of degree $d$ in $\P^4$ or in $\P^5,$ which are exactly those supporting Ulrich bundles of rank $2,$ contains no lines if $d\ge 6$ or $d\ge 8$ respectively. This means that all Ulrich bundles on such hypersurfaces are very ample \cite[Theorem 1]{lopez2023geometrical}.} 
		
		Indeed, as $\L^2\E,$ being isomorphic either to $\O_X(d-1)$ for $n=3$ or to $\E^\ast(3(d-1)/2)$ when $n=4,$ is aCM, we know from \cite[Proposition 5.1]{tripathi2017rank} that $\L^2M_\E$ is aCM as well. Since $H^2(X,\L^2M_\E)=0,$ the first part is provided by Example \ref{ex:2-normality}. Assuming $n=4$ and $r=3,$ by applying \cite[Example B.1.3]{lazarsfeld2017positivity} to the long exact sequence 
		\[
		0\to\L^3M_\E\to\L^3H^0(X,\E)\o\O_X\to\L^2H^0(X,\E)\o\E\to H^0(X,\E)\o\S^2\E\to \S^3\E\to 0,
		\] given by (\ref{eq:fact1}), we deduce that $\E$ is $3$-normal if (and only if) $H^3(X,\L^3M_\E)=0.$ As $\L^3M_\E$ is aCM \cite[Theorem 5.7(a)]{ravindra2019rank}, the claim follows.
	\end{rmk}
	
	The behaviour of $2$-normality of rank $2$ Ulrich bundles on pfaffian surfaces and paffian threefolds or fourfolds is different: while no (very general) pfaffian surface of degree $d\ge 5$ supports a $2$-normal Ulrich bundle of rank $2$ (Proposition \ref{lem:ulrich-tensor-2d-hyper}), all rank $2$ Ulrich bundles on pfaffian $3$-folds and $4$-folds of degree $d\ge3$ are $2$-normal.

	\appendix
	\section{Chern classes computations}
	We calculate the first three Chern classes of $\E\o\E,\S^2\E,\L^2\E$ and the first two Chern classes of $\S^3\E$ for a vector bundle $\E$ on a smooth variety.
	
	\begin{lemma}\label{app:tensor-chern}
		Let $\E$ be a vector bundle of rank $r$ on a smooth variety. Then:
		\begin{itemize}
			\item[(i)] $c_1(\E\o\E)=2rc_1(\E).$
			\item[(ii)] $c_2(\E\otimes\E)=(2r^2-r-1)c_1(\E)^2+2rc_2(\E).$
			\item[(iii)] $c_3(\E\o\E)=\frac{2}{3}(2 r^3-3 r^2-2 r+3)c_1(\E)^3 +(4 r^2-2 r-4)c_1(\E)c_2(\E)+2rc_3(\E).$
		\end{itemize}
	\end{lemma}
	\begin{proof}
		Item (i) is well known, see for instance \cite[Proposition 5.18]{eisenbud20163264}. The Chern character $$\mr{ch}=\rk+c_1+\frac{1}{2}(c_1^2-2c_2)+\frac{1}{6}(c_1^3-3c_1c_2+3c_3)+\cdots$$ satisfies $\mr{ch}(\E)^2=\mr{ch}(\E\o\E)$ \cite[§5.5.2]{eisenbud20163264}, thus (ii) and (iii) are obtained by equating the terms of degree $2$ and $3$ respectively.		
	\end{proof}
	
	\begin{lemma}\label{app:chern-S^k}
		For a vector bundle $\E$ of rank $r$ on a smooth variety we have $$c_1(\S^k\E)=\binom{r+k-1}{k-1}c_1(\E)=\binom{r+k-1}{r}c_1(\E)$$ for all $k\ge1.$
	\end{lemma}
	
	\begin{proof}
		See \cite[p. 523]{rubei2013slope}.
	\end{proof}
	
	\begin{lemma}\label{app:exterior-chern}
		Let $\E$ be a vector bundle of rank $r$ on a smooth variety. Then:
		\begin{itemize}
			\item[(i)] $c_1(\L^2\E)=(r-1)c_1(\E).$
			\item[(ii)] $c_2(\L^2\E)=\binom{r-1}{2}c_1(\E)^2+(r-2)c_2(\E).$
			\item[(iii)] $c_3(\L^2\E)=\binom{r-1}{3}c_1(\E)^3+(r-2)^2c_1(\E)c_2(\E)+(r-4)c_3(\E).$
		\end{itemize}
	\end{lemma}
	\begin{proof} For (i), recall that $\E\o\E\cong\S^2\E\oplus\L^2\E$ and apply Lemmas \ref{app:tensor-chern}(i)--\ref{app:chern-S^k} to the equality $c_1(\E\o\E)=c_1(\S^2\E)+c_1(\L^2\E).$ We prove (ii) and (iii) by induction on $r\ge 1$ by using the splitting principle. Since they are clearly true for $r=1,$ we suppose $r\ge 2$ and that $\E$ decomposes as $\E=L\oplus \F$ with $L$ being a line bundle and $\F$ being a direct sum of $(r-1)$ line bundles. Then, in virtue of $\L^2\E\cong\L^2\F\oplus(L\o\F),$ which yields $c_k(L\o\F)=\sum_{i=0}^k\binom{r-1-k+i}{i}c_1(L)^ic_{k-i}(\F),$ we get
		\begin{align*}
			c_1(\E)=& c_1(L)+c_1(\F), & c_1(\L^2\E)=&c_1(\L^2\F)+c_1(L\o\F),\\
			c_2(\E)=&c_1(L)c_1(\F)+c_2(\F), & c_2(\L^2\E)=&c_2(\L^2\F)+c_1(L\o\F)c_1(\L^2\F)+c_2(L\o\F),\\
			c_3(\E)=&c_1(L)c_2(\F)+c_3(\F), & c_3(\L^2\E)=&c_3(\L^2\F)+c_2(\L^2\F)c_1(L\o\F)+c_1(\L^2\F)c_2(L\o\F)+c_3(L\o\F),
		\end{align*}
		Using the inductive hypothesis on $\F$ we immediately get the desired formulae.				
	\end{proof}
	
	\begin{cor}\label{app:symmetric-chern}
		Let $\E$ be a vector bundle of rank $r$ on a smooth variety. Then:
		\begin{itemize}
			\item[(i)] $c_1(\S^2\E)=(r+1)c_1(\E).$
			\item[(ii)] $c_2(\S^2\E)=\frac{(r+2)(r-1)}{2}c_1(\E)^2+(r+2)c_2(\E).$
			\item[(iii)] $c_3(\S^2\E)=\frac{(r+3)(r-1)(r-2)}{6}c_1(\E)^3+(r^2+2r-4)c_1(\E)c_2(\E)+(r+4)c_3(\E).$
		\end{itemize}
	\end{cor}
	\begin{proof}
		Thanks to the decomposition $\E\o\E\cong\S^2\E\oplus\L^2\E,$ the above formulae immediately follow from $$c_k(\E\o\E)=\sum_{i=0}^kc_i(\S^2\E)c_{k-i}(\L^2\E)$$ and from Lemmas \ref{app:tensor-chern}--\ref{app:chern-S^k}--\ref{app:exterior-chern}.
	\end{proof}
	
	\begin{lemma}\label{app:S^k-S^3}
		Let $\E$ be a vector bundle of rank $r$ on a smooth variety. Then:
		\begin{itemize}
			\item[(i)] $c_1(\S^3\E)=\frac{(r+2)(r+1)}{2}c_1(\E).$
			\item[(ii)] $c_2(\S^3\E)=\frac{1}{8}(r-1)(r+2)(r^2+5r+8)c_1(\E)^2+\frac{1}{2}(r+2)(r+3)c_2(\E).$
		\end{itemize}
	\end{lemma}
	
	\begin{proof}
		Formula (i) is a special case of Lemma \ref{app:chern-S^k}. For (ii), as it is clearly true for line bundles, we proceed by induction on $r\ge2$ using the splitting principle. Then we suppose $\E=L\oplus\F$ where $L$ is a line bundle and $\F$ is sum of $r-1$ line bundles. In this way, we have $$c_1(\E)=c_1(L)+c_1(\F)\ \text{and}\ c_2(\E)=c_1(L)c_1(\F)+c_2(\F).$$ For (ii), expanding through $c_k(\H\oplus\g)=\sum_{i=0}^{k}c_i(\H)c_{k-i}(\g)$ and using the inductive hypothesis together with (i), Corollary \ref{app:symmetric-chern}(i)-(ii) and \cite[Proposition 5.17]{eisenbud20163264}, we have 
		\begin{align*}
			c_2(\S^3\E)&=c_2(\S^3(L\oplus\F))=c_2\left(L^{\o3}\oplus L^{\o2}\o\F\oplus L\o\S^2\F\oplus\S^3\F\right)\\
			&=c_1(L^{\o3})\left[c_1(L^{\o2}\o\F)+c_1(L\o\S^2\F)+c_1(\S^3\F)\right]+c_2(L^{\o2}\o\F)\\
			&\quad+c_1(L^{\o2})\left[c_1(L\o\S^2\F)+c_1(\S^3\F)\right]+c_2(L\o\S^2\F)+c_1(L\o\S^2\F)c_1(S^3\F)+c_2(\S^3\F)\\
			&=3c_1(L)\left(c_1(\F)+2(r-1)c_1(L)+rc_1(\F)+\binom{r}{2}c_1(L)+\binom{r+1}{2}c_1(\F)\right)\\
			&\quad+c_2(\F)+2(r-2)c_1(\F)c_1(L)+4\binom{r-1}{2}c_1(L)^2\\
			&\quad\quad+\left(c_1(\F)+2(r-1)c_1(L)\right)\left(rc_1(\F)+\binom{r}{2}c_1(L)+\binom{r+1}{2}c_1(\F)\right)\\
			&\quad\quad\quad+\frac{1}{2}(r-2)(r+1)c_1(\F)^2+(r+1)c_2(\F)+r\left(\binom{r}{2}-1\right)c_1(F)c_1(L)+\binom{\frac{1}{2}r(r+1)}{2}c_1(L)^2\\
			&\quad\quad\quad\quad+\binom{r+1}{2}c_1(\F)\left(rc_1(\F)+\binom{r}{2}c_1(L)\right)\\
			&\quad\quad\quad\quad\quad+\frac{1}{8}(r-2)(r+1)\left((r-1)^2+5(r-1)+8\right)c_1(\F)^2+\frac{1}{2}(r+1)(r+2)c_2(\F)\\
			&=\frac{1}{8}(r+2)(r^3+4r^2+3r-8)(c_1(L)^2+c_1(\F)^2)\\
			&\quad+\frac{1}{4}(r+2)(r^3+4r^2+5r-2)c_1(L)c_1(\F)+\frac{1}{2}(r+2)(r+3)c_2(\F)\\
			&=\frac{1}{8}(r-1)(r+2)(r^2+5r+8)(c_1(L)^2+2c_1(L)c_1(\F)+c_1(\F)^2)\\
			&\quad+\frac{1}{2}(r+2)(r+3)(c_1(L)c_1(\F)+c_2(\F))\\
			&=\frac{1}{8}(r-1)(r+2)(r^2+5r+8)c_1(\E)^2+\frac{1}{2}(r+2)(r+3)c_2(\E)
		\end{align*}
		as required.
	\end{proof}

\end{document}